\documentclass[12pt]{amsart}
\usepackage{amssymb}
\usepackage{graphicx}
\usepackage{times}
\usepackage{graphicx}

\newtheorem{thm}{Theorem}[subsection]

\newtheorem{lem}[thm]{Lemma}

\newtheorem*{nwtthm}{No Wandering Triangles Theorem}
\newtheorem*{mathe}{Main theorem}
\theoremstyle{definition}
\newtheorem{defn}[thm]{Definition}

\theoremstyle{remark}
\newtheorem{rem}[thm]{Remark}
\numberwithin{equation}{section}
\numberwithin{figure}{subsection}

\newcommand{\A}{\mathcal A}
\newcommand{\B}{\mathcal B}
\newcommand{\C}{\mathcal C}
\newcommand{\AP}{\mathcal A\mathcal P}
\newcommand{\qml}{\mathrm{QML}}
\newcommand{\WT}{\mathcal W\mathcal T}
\newcommand{\bLa}{\widehat{\mathcal L}}
\newcommand{\La}{\mathcal L}
\newcommand{\lam}{\mathcal L}

\newcommand{\ma}{\mathbf{a}}

\newcommand{\mc}{\mathbf{c}}
\newcommand{\md}{\mathbf{d}}
\newcommand{\mg}{\mathbf{g}}

\newcommand{\mell}{ {\boldsymbol \ell} }
\newcommand{\mm}{\mathbf{m}}
\renewcommand{\mp}{\mathbf{p}}
\newcommand{\mq}{\mathbf{q}}
\newcommand{\mr}{\mathbf{r}}
\newcommand{\ms}{\mathbf{s}}
\newcommand{\mt}{\mathbf{t}}

\newcommand{\mx}{\mathbf{x}}
\newcommand{\my}{\mathbf{y}}
\newcommand{\mz}{\mathbf{z}}

\newcommand{\T}{\mathbf{T}}

\newcommand{\complex}{\mathbb C}
\newcommand{\rsphere}{\widehat\complex}
\newcommand{\disk}{\mathbb D}
\newcommand{\Pro}{\mathcal P}
\newcommand{\reals}{\mathbb R}
\newcommand{\uc}{\mathbb S ^1}
\newcommand{\ints}{\mathbb Z}

\newcommand{\ch}{\mbox{$\mathrm{Ch}$}}

\DeclareMathOperator{\dom}{Dom}
\DeclareMathOperator{\itin}{itin}
\DeclareMathOperator{\monitin}{itin_{mon}}
\newcommand{\mon}{\mathrm{mon}}

\newcommand{\ph}{\varphi}
\newcommand{\si}{\sigma}
\newcommand{\Ta}{\Theta}
\newcommand{\ta}{\theta}
\newcommand{\al}{\alpha}
\newcommand{\be}{\beta}

\newcommand{\bi}{\itin_{\text{max}}}
\newcommand{\ti}{\itin_{\text{true}}}

\newcommand{\ol}{\overline}
\newcommand{\sh}{\ol{\text{sh}}}
\newcommand{\lo}{\ol{\text{lo}}}
\newcommand{\sm}{\setminus}

\begin{document}

\title{Cubic Critical Portraits and Polynomials with Wandering Gaps}

\author{Alexander~Blokh} \thanks{The authors were partially supported
  by NSF grant DMS-0901038 (A.B.) and by NSF grant DMS-0906316 (L.O.)}

\author{Clinton~Curry}

\author{Lex~Oversteegen}

\address[Alexander~Blokh, Lex~Oversteegen and Clinton~Curry]
{Department of Mathematics\\ University of Alabama at Birmingham\\
  Birmingham, AL 35294-1170\\ Tel.: (205)934-2154\\ Fax:
  (205)934-9025}

\email[Alexander~Blokh]{ablokh@math.uab.edu}
\email[Lex~Oversteegen]{overstee@math.uab.edu}
\email[Clinton~Curry]{clintonc@math.sunysb.edu}

\date{March 9, 2010; revised versions December 12, 2010, July 20, 2011 and December 14, 2011}


\begin{abstract}
  Thurston introduced $\si_d$-invariant laminations (where $\si_d(z)$
  coincides with $z^d:\uc\to \uc$, $d\ge 2$) and defined
  \emph{wandering $k$-gons} as sets $\T\subset \uc$ such that
  $\si_d^n(\T)$ consists of $k\ge 3$ distinct points for all $n\ge 0$
  and the convex hulls of all the sets $\si_d^n(\T)$ in the plane are
  pairwise disjoint. He proved that $\si_2$ has no wandering $k$-gons.

  Call a lamination with wandering $k$-gons a \emph{WT-lamination}.
  In a recent paper it was shown that uncountably many cubic
  WT-laminations, with pairwise non-conjugate induced maps on the
  corresponding quotient spaces $J$, are realizable as cubic
  polynomials on their (locally connected) Julia sets. Here we use a
  new approach to construct cubic WT-laminations with the above
  properties so that any wandering branch point of $J$ has a dense
  orbit in each subarc of $J$ (we call such orbits \emph{condense}),
  and show that critical portraits corresponding to such laminations
  are dense in the space $\A_3$ of all cubic critical portraits.
\end{abstract}

\maketitle

\section{Introduction}\label{intro}

\subsection{Preliminaries}
Let $\complex$ be the complex plane and $\rsphere = \complex \cup
\{\infty\}$ be the complex sphere.  Theorem~\ref{t:nowanpol} is a
special case of a theorem of Thurston \cite[Theorem II.5.2]{thu85}.

\begin{thm}[No wandering vertices for quadratics]\label{t:nowanpol}
  Let $P(z) = z^2 + c$ be a polynomial which has connected Julia set
  $J_P$.  Then, if $z_0 \in J_P$ is a point such that $J_P \setminus
  \{z_0\}$ has at least three components, then $z_0$ is either
  preperiodic or eventually maps to the critical point $0$.
\end{thm}

A point $z_0 \in J_P$ is called a \emph{vertex} if $J_P \setminus
\{z_0\}$ has at least three components, and a vertex is called
\emph{wandering} if it is not periodic and never maps to a critical
point.  It is shown in \cite[Theorem 1.1]{bo04} that there exist
uncountably many cubic polynomials, each of which has a locally
connected Julia set and a wandering vertex.  Here we improve on these
examples by constructing uncountably many cubic polynomials with locally connected
Julia sets whose wandering vertices have \emph{condense orbits}, where
a set $A \subset X$ is called \emph{condense} if $A$ intersects every
non-degenerate subcontinuum of $X$.  This is much stronger than the
density of the orbit of $z_0$ in $J$ (e.g., most subcontinua of
dendritic Julia sets are nowhere dense); see \cite{bco11} for some
consequences of having a condense orbit. To state our results
precisely, we first briefly describe Thurston's theory of invariant
laminations.

Laminations were introduced by Thurston \cite[Definition
II.4.2]{thu85} as a tool for studying individual complex polynomials
$P:\rsphere \to \rsphere$ and their parameter space. Let $P$ be a
degree $d$ polynomial with a connected Julia set $J_P$. Its filled-in
Julia set $K_P$ is compact, connected, and full, so its complement
$\rsphere \setminus K_P$ is conformally isomorphic to the open unit
disk $\disk$. By \cite[Theorem 9.5]{milnor} one can choose a conformal
isomorphism $\Psi:\disk \to \rsphere \setminus K_P$ so that $\Psi$
satisfies $\Psi(z^d) = (P|_{\rsphere \setminus K_P} \circ \Psi)(z)$
for all $z \in \disk$. For a locally connected Julia set $J_P$, the
map $\Psi$ extends to a continuous map $\ol{\Psi}: \ol{\disk}\to
\ol{\rsphere\setminus K_P}$ \cite[Theorem 17.14]{milnor},
semiconjugating $z \mapsto z^d$ on $\ol{\disk}$ to $P|_{\ol{\rsphere
    \setminus K_P}}$. Let $\psi:\uc \to J_P$ denote the restriction
$\ol \Psi |_{\uc}$, and let $\si_d:\uc \to \uc$ denote the map $z
\mapsto z^d$. Define an equivalence relation $\sim_P$ on $\uc$ so that
$x \sim_P y$ if and only if $\psi(x)=\psi(y)$, and call it the
\emph{$\si_d$-invariant lamination generated by $P$}. The quotient
space $\uc/{\sim_P}=J_{\sim_P}$ is homeomorphic to $J_P$ and the
\emph{induced map} $f_{\sim_P}:J_{\sim_P}\to J_{\sim_P}$ defined by
$f_{\sim_P} = \psi \circ \si_d \circ \psi^{-1}$ is conjugate to
$P|_{J_P}$. One can introduce abstract laminations (frequently denoted
by $\sim$) as equivalence relations on $\uc$ similar to laminations
generated by polynomials as above (see Section 2). We call $J_\sim =
\uc/\sim$ a \emph{topological Julia set}, and denote the map
\emph{induced} by $\si_d$ on $J_\sim$ by $f_\sim$.

Laminations are also used in studying the space $\Pro_d\cong
\complex^{d-1}$ of degree $d\ge 2$ monic centered polynomials
$z\mapsto z^d+a_{d-2}z^{d-2}+\dots+a_0$. The \emph{connectedness locus
  of $\Pro_d$} is the set $\C_d$ of parameters in $\Pro_d$ for which
the Julia set is connected (which is a continuum by \cite[Corollary
3.7]{bh88} and \cite{lav89}). The set $\C_2$ is called the
\emph{Mandelbrot set} and is denoted by $\mathcal{M}$. Thurston
\cite[Definition II.6.9]{thu85} defined a ``meta-lamination'',
referred to as $\qml$, and showed that the closure of the space of
all $\si_2$-invariant (\emph{quadratic}) laminations can be thought
of as the quotient space $\uc/\qml=\mathcal{M}_c$, a locally
connected continuum that serves as a combinatorial model of the
boundary of $\mathcal{M}$. The exact relationship between
$\mathcal{M}_c$ and $\mathcal{M}$ is unknown. Thurston conjectured
that the boundary of $\mathcal{M}$ is homeomorphic to
$\mathcal{M}_c$, which, if true, would imply that $\mathcal{M}$ is
locally connected. A crucial role in his study is played by the next
theorem.

\begin{nwtthm}[{\cite[Theorem II.5.2]{thu85}}]
  Let $\sim$ be a $\si_2$-invariant lamination.  If $\mg$ is a
  $\sim$-class of cardinality at least three, then $\mg$ is either
  eventually critical or eventually periodic.
\end{nwtthm}

Thus, branch points in degree 2 topological Julia sets either are
precritical or preperiodic, which can be regarded as a generalization
of the corresponding property of continuous maps of finite graphs.
Also, it follows from the No Wandering Triangles Theorem that branch
points of $\mathcal{M}_c$ correspond to laminations with
preperiodic critical classes. This makes the problem of extending
Thurston's result to higher degrees, posed by Thurston in
\cite{thu85}, important.  Indeed, J. Kiwi~\cite[Theorem 1.1]{kiwigap}
answers this call by showing that a wandering non-precritical gap in a
$\si_d$-invariant lamination is at most a $d$-gon. Thus all infinite
$\sim$-classes and Jordan curves in $J_\sim$ are preperiodic.  Later
on \cite[Theorem B]{blolev99} it was shown that if $\Gamma$ is a
non-empty collection of wandering non-precritical $d_j$-gons
($j=1,2,\dots$) of a $\si_d$-invariant lamination with distinct grand
orbits, then $\sum_{\Gamma} (d_j - 2)\le d-2$. Thus, there are bounds
on the number of wandering gaps with distinct grand orbits.

However, even for $\si_3$-invariant (henceforth \emph{cubic}) laminations
wandering triangles exist: by \cite{bo04}, Theorem 1.1, there are
uncountably many pairwise non-conjugate cubic polynomials $P$ which
have dendritic Julia sets with a wandering branch point. Hence for
each such polynomial $P$ the corresponding lamination $\sim_P$ has a
wandering triangle. Since the construction in \cite{bo04} is quite
specific, the corresponding ``wandering'' dynamics might be rare.

This paper gives a more general and flexible construction than that in
\cite{bo04}, extending the above. Let $\A_3$ be the space of all cubic
critical portraits. We construct cubic laminations such that the
forward orbit of a wandering triangle is \emph{dense in the entire
  lamination} (this is much stronger than in \cite{bo04}), and prove
that \emph{their critical portraits form a locally uncountable and
  dense subset of $\A_3$}. Thus, critical portraits with wandering
vertices in their Julia sets are not rare.  Even more, the topological
polynomial of each of these laminations is conjugate to a polynomial
restricted to its Julia set by \cite{bco11}.

In conclusion we want to thank the referees for careful reading of the
manuscript and useful remarks which helped us improve the paper.

\subsection{Statement of the results}\label{results}

We parameterize the circle as $\uc = \reals / \ints$ with total
arclength $1$. In this parameterization the map $\si_d$ corresponds to
the map $t\mapsto d\cdot t \mod 1$. The \emph{positive} direction on
$\uc$ is \emph{counterclockwise}, and by an arc $(p,q)$ in $\uc$ we
mean the positively oriented arc from $p$ to $q$. A \emph{monotone}
map $g:(p, q)\to \uc$ is a map such that for each $y\in \uc$ the set
$g^{-1}(y)$ is connected. A monotone map is called \emph{strictly
  monotone} if it is one-to-one. Given a set $A \subset \uc$, we
denote the cardinality of $A$ by $|A|$ and the convex hull of $A$ in
the closed unit disk by $\ch(A)$ (for our purposes it does not matter
whether we use the Euclidian or the hyperbolic metric). In what
follows, given a map $g$, a \emph{($g$-)image} of a set $A$ is the
image of $A$ under an iterate of $g$ while the \emph{first
  ($g$-)image} of a set $A$ is the set $g(A)$. Similar language will
be used for preimages.

Given an equivalence relation $\sim$ on $\uc$, the equivalence classes
are called \emph{($\sim$-)classes} and are denoted by boldface
letters. A $\sim$-class consisting of two points is called a
\emph{leaf}; a class consisting of at least three points is called a
\emph{gap} (this is more restrictive than Thurston's definition in
\cite{thu85}).

Fix an integer $d>1$. Then an equivalence relation $\sim$ is said to
be a \emph{$\si_d$-invariant lamination} if:

\begin{itemize}
\item[(1)] $\sim$ is \emph{closed}: the graph of $\sim$ is a closed
  set in $\uc \times \uc$;
\item[(2)] $\sim$-classes are \emph{pairwise unlinked}: if $\mg_1$
  and $\mg_2$ are distinct $\sim$-classes, their convex hulls
  $\ch(\mg_1), \ch(\mg_2)$ in the closed unit disk $\ol{\disk}$ are disjoint;
\item[(3)] $\sim$-classes are \emph{totally disconnected} (and hence
  $\sim$ has uncountably many classes) provided $\uc$ is not one
  class;
\item[(4)] $\sim$ is \emph{forward invariant}: for a class $\mg$, the
  set $\si_d(\mg)$ is also a class;
\item[(5)] $\sim$ is \emph{backward invariant}: for a class $\mg$, its
  first preimage $\si_d^{-1}(\mg)=\{x\in \uc: \si_d(x)\in \mg\}$ is a union
  of classes; and
\item[(6)] $\sim$ is \emph{gap invariant}: for any gap $\mg$, the map
  $\si_d|_{\mg}: \mg\to \si_d(\mg)$ is a \emph{covering map with
    positive orientation}, i.e., for every connected component $(s,
  t)$ of $\uc\setminus \mg$ the arc $(\si_d(s), \si_d(t))$ is a
  connected component of $\uc\setminus \si_d(\mg)$.
\end{itemize}
Notice that (3) and (5) are implied by (4).

Call a class $\mg$ \emph{critical} if $\si_d|_{\mg}: \mg\to
\si_d(\mg)$ is not one-to-one, and \emph{precritical} if
$\si_d^j(\mg)$ is critical for some $j\ge 0$. Call $\mg$
\emph{preperiodic} if $\si^i_d(\mg)=\si^j_d(\mg)$ for some $0\le
i<j$. A gap $\mg$ is \emph{wandering} if $\mg$ is neither preperiodic
nor precritical, and a lamination which has a wandering gap is called
a \emph{WT-lamination}. For a lamination $\sim$, let $J_\sim = \uc /
\sim$, and let $\pi_\sim: \uc \to J_\sim$ be the corresponding
quotient map.  Then the map $f_\sim:J_\sim \to J_\sim$ defined by
$f_\sim = \pi_\sim \circ \si_d \circ \pi_\sim^{-1}$ is the map
\emph{induced} on $J_\sim$ by $\si_d$ (the map $f_\sim$ is
well-defined in view of (4)). We call $f_\sim$ a \emph{topological
  polynomial}, and $J_\sim$ a \emph{topological Julia set}.

Though we define laminations as specific equivalence relations on
$\uc$, one can also work with a corresponding collection of chords,
called a \emph{geometric ($\si_d$-invariant) lamination}. Given a
$\si_d$-invariant lamination $\sim$, its \emph{geometric lamination}
$(\La_\sim, \si_3)$ is defined as the union of all chords in the
boundaries of convex hulls of all $\sim$-classes; the map $\si_3$ is
then extended over $\La_\sim$ by linearly mapping each chord in
$\La_\sim$ forward. Clearly, $\La_\sim$ is a closed family of chords
of $\ol{\disk}$ (in the above situation we also include degenerate
$\sim$-classes in the list of chords). Geometric laminations have
properties similar to the properties of the laminations introduced
by Thurston \cite[Definition II.4.2]{thu85}. One of the main ideas
of this paper is to study \emph{finite truncations of geometric
laminations $(\La_\sim, \si_3)$
  defined up to an order preserving conjugacy} and considering such
\emph{finite laminations} as purely combinatorial objects.

\subsubsection{Critical portraits}

Fix $d\ge 2$. A key tool in studying $\C_d$ is \emph{critical
  portraits}, introduced in \cite{fish89}, and widely used afterward
\cite{bfh92, golmil93,ki4,po93,po09}. We now recall some standard
material.  Here we follow \cite[Section 3]{ki4} closely.  Call a chord
of the circle, with endpoints $a, b\in \uc$, \emph{critical} if
$\si_d(a)=\si_d(b)$.

\begin{defn}\label{cp}
  Fix $d\ge 2$. A \textbf{($\si_d$-)critical portrait} is a collection
  $\Ta=\{\Ta_1, \dots, \Ta_n\}$ of finite subsets of $\uc$ such that
  the following hold.
  \begin{enumerate}
  \item The boundary of the convex hull $\ch(\Ta_i)$ of each set
    $\Ta_i$ consists of critical chords;
  \item the sets $\Ta_1, \dots, \Ta_n$ are \textbf{pairwise unlinked}
    (that is, their convex hulls are pairwise disjoint); and
  \item $\sum (|\Ta_i|-1)=d-1$.
  \end{enumerate}
\end{defn}

To comment on Definition~\ref{cp} we need the following terminology:
the sets $\Ta_1, \dots, \Ta_n$ are called the \emph{initial sets of
  $\Ta$} (or \emph{$\Ta$-initial sets}). The convex hulls of the
$\Ta$-initial sets divide the rest of the unit disk $\disk$ into
components. Consider one such component, say $U$. Then $\partial U$
consists of circular arcs and critical chords $\ell_1, \dots,
\ell_k$. If we extend $\si_d$ linearly on the chords in the
boundary of $U$, then, by Definition~\ref{cp}, $\partial U$ maps onto
$\uc$ one-to-one except for the collapsing critical chords $\ell_1,
\dots, \ell_k$.  In fact, Definition~\ref{cp} is designed to achieve
this dynamical property related to situations of the following
kind. Suppose that $P$ is a polynomial of degree $d$ with dendritic
(i.e., locally connected and containing no simple closed curve) Julia
set $J_P$ such that all its critical values are endpoints of $J_P$.
Then the arguments of rays landing at the critical points of $P$ form
the initial sets of a certain critical portrait associated with this
polynomial.

Let $\Ta$ be a critical portrait. Denote by $A(\Ta)$ the union of
all angles from the initial sets of $\Ta$. As was remarked above,
the convex hulls of the $\Ta$-initial sets divide the rest of the
unit disk $\disk$ into components. According to Definition~\ref{ue},
points of $\uc\setminus A(\Ta)$ belonging to the boundary of one
such component will be declared equivalent. However, since we need
this notion in more general situations, we will not assume in
Definition~\ref{ue} that $\Ta$ is a critical portrait.

\begin{defn}[{\cite{bfh92}, \cite{fish89}, \cite{golmil93}, \cite[Definition
3.4]{ki4}}]\label{ue}
Let $\Ta=\{\Ta_1, \dots, \Ta_n\}$ be a finite collection of pairwise
unlinked finite subsets of $\uc$ (not necessarily a critical
portrait). Angles $\alpha, \beta\in \uc\sm A(\Ta)$ are
\textbf{$\Ta$-unlinked equivalent} if $\{\alpha, \beta\},$ $\Ta_1,$
$\dots,$ $\Ta_n$ are pairwise unlinked (i.e., if the chord
$\ol{\al\be}$ is disjoint from the $\bigcup_{i=1}^k \ch(\Ta_k)$). The
classes of equivalence $L_1(\Ta),$ $\dots,$ $L_d(\Ta)$ are called
\textbf{$\Ta$-unlinked classes}. Each $\Ta$-unlinked class $L$ is the
intersection of $\uc\sm A(\Ta)$ with the boundary of a component of
$\disk\sm \bigcup \ch(\Ta_i)$. If $\Ta$ is a degree $d$ critical
portrait, then each $\Ta$-unlinked class of $\Ta$ is the union of
finitely many \emph{open (in $\uc$)} arcs of total length $1/d$. Thus,
in this case there are precisely $d$ $\Ta$-unlinked classes.
\end{defn}

We introduce the following topology on the set of critical
portraits.

\begin{defn}[compact-unlinked topology {\cite[Definition
    3.5]{ki4}}]\label{cu}
  Define the space $\A_d$ as the set of all degree $d$ critical
  portraits.  We endow it with the \textbf{compact-unlinked topology}:
  the coarsest topology on $\A_d$ such that, for any compact set $X
  \subset \uc$, the set of critical portraits whose critical leaves
  are unlinked with $X$ is open.
\end{defn}

For example, $\A_2$ is the quotient of the circle with antipodal
points identified, so it is homeomorphic to $\uc$. As another example,
take a cubic critical portrait which consists of a triangle $T$ with
vertices $a, b, c$ and three critical edges, and choose compact sets
$X_1$, $X_2$, and $X_3$ in the three components of $\uc \setminus \{a,
b, c\}$.  Then every neighborhood of $T$ includes critical triangles
close to $T$ and pairs of disjoint critical leaves each of which is
close to an edge of $T$.

For a critical portrait $\Ta$, a lamination $\sim$ is called
\emph{$\Ta$-compatible} if all $\Ta$-initial sets are subsets of
$\sim$-classes. The trivial lamination which identifies all points of
the circle is compatible with any critical portrait.  If there is a
$\Ta$-compatible WT-lamination, $\Ta$ is called a \emph{critical
  WT-portrait}.

An important tool for describing the dynamics of a lamination is the
itinerary.  For simplicity and because it suits our purpose, the
following is defined for critical portraits consisting of $d-1$
disjoint critical leaves.  For a critical portrait $\Ta$ and a point
$t \in \uc$, we define two types of itineraries: \emph{one-sided
  itineraries}, denoted $\itin^+(t)$ and $\itin^-(t)$ which are
sequences of $\Ta$-unlinked classes (see \cite[Definition 3.13]{ki4}),
and the itinerary $\itin(t)$ corresponding to the Markov partition
into $\Ta$-unlinked classes and elements of $\Ta$.  We define
$\itin^+(t) = i_0^+i_1^+i_3^+\ldots$ where $i_k^+$ is the
$\Ta$-unlinked class which contains the interval $(\si_d^k(t),
\si_d^k(t)+\epsilon)$ for some small $\epsilon>0$; $\itin^-(t) =
i_0^- i_1^- \ldots $ is similarly defined.  We define $\itin(t) =
i_0i_1i_2, \ldots$ such that $i_k$ is the $\Ta$-unlinked class
or critical leaf containing $\si_d^k(t)$.  Thus, if $\si_d^k(t)$ is
not an endpoint of a critical leaf, then $i_k$, $i^+_k$, and $i^-_k$
are all equal, and if $\si_d^k(t)$ is an endpoint of a critical leaf,
then $i_k$, $i^+_k$, and $i^-_k$ are all different.

An important class of critical portraits are \emph{critical portraits
  with aperiodic kneading}.
A critical portrait $\Ta$ with $d-1$ critical leaves has aperiodic
kneading if for each angle $\ta\in A(\Ta)$ the itineraries $\itin^+(\ta)$
and $\itin^-(\ta)$ are not periodic \cite[Definition 4.6]{ki4}. The family
of all critical portraits with $d-1$ critical leaves and aperiodic
kneading sequence is denoted by $\AP_d$.

\begin{defn}[Definition 4.5 \cite{ki4}]\label{lamtheta}
  Let $\Ta$ be a critical portrait with aperiodic kneading.  The
  lamination $\sim_\Ta$ is defined as the smallest equivalence
  relation such that if $\itin^+(x)=\itin^-(y)$ then $x$ and $y$ are
  $\sim_\Ta$-equivalent; it is said to be the lamination
  \emph{generated} by $\Ta$.
\end{defn}

Now we quote a fundamental result of Kiwi \cite{ki4} (see
also \cite{kiwi97}). To state it we need the following definitions.
A map $\ph: X\to Y$ from a continuum $X$ to a continuum $Y$ is said
to be \emph{monotone} if $\ph$-preimages of points are continua. A
\emph{dendrite} is a locally connected non-degenerate continuum
which contains no subsets homeomorphic to $\uc$.

\begin{thm}[\cite{kiwi97, ki4}]\label{kiwistruct}
  Let $\Ta\in \AP_d$.  Then the lamination $\sim_\Ta$ is the unique
  $\Ta$-compatible $\si_d$-invariant lamination, $J_{\sim_\Ta}$ is a
  dendrite, all $\sim_\Ta$-classes are finite, and there exists a
  polynomial $P$ whose Julia set $J_P$ is a non-separating continuum
  in the plane such that $P|_{J_P}$ is semiconjugate to
  $f_{\sim_\Ta}|_{J_{\sim_\Ta}}$ by a monotone map $\psi:J_P \to
  J_{\sim_\Ta}$. For each $P$-preperiodic point $x\in J_P$ the set
  $\psi^{-1}(\psi(x))=\{x\}$. Furthermore, $J_P$ is locally connected
  at preperiodic points.
\end{thm}

\begin{rem}
  This theorem is an amalgamation of several results of Kiwi;  we
  will describe those results here.  Suppose that $\Ta$ satisfies the
  hypothesis.  Then, by \cite[Corollary 3.26]{ki4}, there is a
  polynomial $P$ with connected Julia set whose critical portrait is
  $\Ta$.  By \cite[Theorem 1]{ki4}, all cycles of $P$ are repelling,
  and the real lamination of $P$ is the unique $\Ta$-compatible
  lamination.  By \cite[Theorem 5.12]{kiwi97}, $P|_{J_P}$ is
  monotonically semiconjugate to $f_|{\sim_\Ta}$; since $J_P$ is
  non-separating, it follows that $J_{\sim_\Ta}$ is a dendrite.  That
  the monotone projection $\psi:J_P \to J_{\sim_\Ta}$ is one-to-one
  at   $P$-preperiodic points follows from \cite[Theorem 2]{kiwi97},
  and that $J_P$ is locally connected at $P$-preperiodic points
  follows from \cite[Corollary 1.2]{kiwi97}.
\end{rem}

\subsubsection{Condensity and main results}\label{ss:condens}

For a topological space $X$, a set $A\subset X$ is
\emph{con\-tinuum-dense} (briefly, \emph{condense}) in $X$ if $A\cap B\ne
\emptyset$ for each non-degenerate continuum $B\subset X$. If $X$ is a
dendrite, $A$ is condense in $X$ if and only if $A$ intersects every
open arc. A condense set $A$ is dense in $X$, but condensity is
stronger than density; for example, the set of endpoints of a non-interval
dendritic Julia set $J$ is residual, dense and disjoint from all
non-degenerate open arcs in $J$. If $x$ is a point with condense
orbit in the Julia set $J_P$ of a polynomial $P$, then $J_P$ is a
dendrite or a Jordan curve \cite[Theorem 1.4]{bco11}.

Let $\WT_3$ be the family of all cubic critical WT-portraits.

\begin{mathe}
  For each open $U\subset \A_3$ there is an uncountable set $\B
  \subset U\cap \AP_d \cap \WT_3$ such that the following facts hold:

  \begin{enumerate}

  \item there exists a wandering branch point in $J_{\sim_\Ta}$ whose
    orbit is condense in $J_{\sim_\Ta}$;

  \item the maps $\{f_{\sim_\Ta}|_{J_{\sim_\Ta}}\,\mid\, \Ta\in \B\}$ are
    pairwise non-conjugate;

  \item for each $\Ta\in \B$ there exists a cubic polynomial $P_\Ta$ such
    that $P_\Ta|_{J_{P_\Ta}}$ is conjugate to
    $f_{\sim_\Ta}|_{J_{\sim_\Ta}}$.

  \end{enumerate}

\end{mathe}

\section{Combinatorial Construction}\label{combconstr}

Let us recall that by a \emph{cubic WT-lamination} we mean a
$\si_3$-invariant lamination which admits a wandering (i.e.,
non-precritical and non-preperiodic) gap.  By \cite[Theorem
B]{blolev99}, a wandering gap in such a lamination is a triangle. We
construct a cubic WT-lamination by means of a sequence $((\La_i,
g_i))_{i=1}^\infty$ of \emph{finite cubic critical laminations}, where
$(\La_{i+1}, g_{i+1})$ \emph{continues} $(\La_i, g_i)$ (the
definitions are given below). There is a limit lamination
$\La=\bigcup^\infty_{i=1} \La_i$ with a map $g$ defined on $\bigcup
\La=\bigcup^\infty_{i=1} (\bigcup \La_i)$ so that $g$ and $g_i$ agree
on $\La_i$ for each $i$. By Theorem 3.3 of \cite{bo04}, $\bigcup
\La_i$ can be embedded in $\uc$ by means of an order preserving map so
that the map induced on its image can be extended to $\si_3$.  This
will give rise to the desired invariant cubic WT-lamination.

Unlike \cite{bo04} (where the construction was rather specific), we
are concerned not only with the existence of wandering triangles, but
also with how complicated and dense their dynamics can be. In addition, we
investigate how common WT-critical portraits are in $\A_3$. To
address these issues, we develop a construction with new features.
Compared to \cite{bo04} they can be summarized as follows: (1)
\emph{all preimages of both critical leaves} are represented in
$((\La_i, g_i))_{i=1}^\infty$, and the wandering triangle approaches
them all (this part is responsible for showing that the corresponding
quotient spaces -- i.e., topological Julia sets -- are dendrites with
wandering non-precritical branch points whose orbits are
\emph{condense}); (2) the initial segments of the constructed
laminations can be chosen arbitrarily close to any given finite
lamination (this part is responsible for the density of WT-critical
portraits in $\A_3$).

\subsection{Finite laminations}\label{finlam}

In this section we study \emph{finite cubic laminations}, in
particular \emph{finite cubic WT-laminations} modeling a
$\si_3$-invariant geometric lamination with a wandering triangle at a
finite step of its orbit. For instance, if $\sim$ is a cubic
invariant lamination with a wandering triangle $\T_1$ and critical
leaves $\mc$ and $\md$, then a collection consisting of $\mc$, $\md$,
finitely many of their images and preimages, and finitely many images
of $\T_1$, form a finite cubic WT-lamination (where $g$ is the
restriction of $\si_3$).

\begin{defn}[Finite laminations]\label{finlamdef}
  We will be interested in three levels of specialization for finite
  laminations.
  \begin{description}
  \item[Finite lamination] A \emph{finite lamination} is a finite
    collection $\La$ of finite, pairwise unlinked subsets of $\uc$.
    The elements of $\La$ are called \emph{$\La$-classes}, and the
    union of points of all $\La$-classes is the \emph{basis} of $\La$.
    By a \emph{sublamination} of a finite lamination $\La$ we mean a
    subcollection of $\La$-classes.
  \item[Dynamical lamination] A \emph{dynamical lamination} is a pair
    $(\La, g)$ where $\La$ is a finite lamination, $g$ is a map
    defined on the basis of a finite sublamination $\bLa$ of $\La$
    which can be extended to a covering map of $\uc$ of degree $3$
    that maps $\bLa$-classes onto $\La$-classes.  Note that the pair
    $(\La, g)$ determines $\bLa$ since by definition
    $\dom(g)$ is the basis of $\bLa$.  Any class on which $g$
    is not defined is called a \emph{last class} of $\La$.
  \item[Critical lamination] A dynamical lamination $(\La, g)$ is
    \emph{critical} if there are two-point classes $\mc$ and $\md$ of
    $\La$, called the \emph{critical leaves} of $\La$, modeling a
    cubic critical portrait: $g(\mc)$ and $g(\md)$ are points, and
    every $\{\mc,\md\}$-unlinked class, intersected with the domain of
    $g$, maps forward in an order-preserving fashion.
  \end{description}
\end{defn}

To establish similarity among finite laminations we need
Definition~\ref{isom}.

\begin{defn}\label{isom}
  A bijection $h:X\to Y$ between sets $X, Y\subset \uc$ is an
  \emph{order isomorphism} if it preserves circular order. If $A, A'
  \subset \uc$, we say maps $f:A\to A$ and $g:A'\to A'$ are \emph{conjugate}
  if they are conjugate in a set-theoretic sense by an order
  isomorphism $h:A\to A'$.
  Note that we do
  not require that an order isomorphism is continuous.
  Finite (dynamical) laminations $\La$ and
  $\La'$ are \emph{order isomorphic} if there is a (conjugating,
  respectively) order isomorphism between $\bigcup \La$ and $\bigcup
  \La'$ that sends $\La$-classes to $\La'$-classes.
\end{defn}

$\La$-classes could be of three distinct types.

\begin{defn}\label{leavgap}
  For a finite lamination $\La$, we refer to $\La$-classes consisting
  of one point as \textbf{buds}, of two points as \textbf{leaves}, of
  3 points as \textbf{triples (triangles)}, and to \emph{all}
  $\La$-classes consisting of more than $2$ points as \textbf{gaps}.
\end{defn}

We identify classes with their convex hulls in $\ol{\disk}$.  The
\emph{convex hulls} of leaves, triples (triangles), and gaps are also
called \emph{leaves, triples (triangles), and gaps.}  The boundary
chords of a non-degenerate class are called \emph{edges}.  In
particularly, every leaf is an edge.  We often talk of $\bLa$ when we
really mean $\bigcup \bLa$ and regard the map $g$ as mapping the convex
hull of a class of $\bLa$ to the convex hull of a class of $\La$.
Chords of $\disk$ are denoted $\ol{a}, \ol{b}$ etc; a chord with endpoints
$x, y$ is denoted by $\ol{xy}$.

Given an interval $(p,q) \subset \uc$, a \emph{strictly monotonically
  increasing} map $g:(p, q)\to \uc$ is a strictly monotone map which
preserves circular orientation. For $A\subset \uc$, a map $f:A\to
f(A)$ is \emph{of degree $3$} if there are $x_0<x_1<x_2< x_3=x_0$ with
$f|_{A\cap [x_i,x_{i+1})}$ strictly monotonically increasing for all
$i=0,1,2$ and there are no two points with the same property.

\begin{defn}\label{finwtdef}
  A critical lamination $(\La, g)$, which contains a designated triple
  $\T_1$ as a class, is called a \textbf{finite cubic WT-lamination}
  if:

  \begin{enumerate}
  \item $(\La, g)$ is of degree 3;
  \item $\La$ has a pair of disjoint critical leaves $\C=\{\mc, \md\}$;
  \item each class $\ma \in \La$ satisfies \emph{exactly one} of the
    following:
    \begin{itemize}
    \item $\ma$ is a bud, in which case $\ma=g^j(\mc)$ for some $j$ or
      $\ma=g^k(\md)$ for some $k$ (but not both),
    \item $\ma$ is a leaf such that $g^i(\ma)$ is a critical leaf for
      some $i$, or
    \item $\ma$ is a triple such that $\ma=g^n(\T_1)$ for some
      $n$
    \end{itemize}
    (so leaves in $\La$ are pullbacks of $\mc$ or $\md$, buds are
    images of $\mc$ or $\md$, and triangles are images of $\T_1$);
  \item all classes from the grand $g$-orbits of $\T_1,$ $\mc,$ and
    $\md$ are unlinked.
  \end{enumerate}

  We will always denote the $n^{th}$ triple in the orbit of $\T_1$ as $\T_n = g^{n-1}(\T_1)$.
  The last classes of $\La$, which are eventual images of $\T_1, \mc$
  and $\md$, are denoted by $\T^l(\La)=\T^l, \mc'(\La)=\mc'$ and
  $\md'(\La)=\md'$; we call $\T^l$ the \textbf{last triple of $\La$}
  (here, the superscript $l$ stands for ``last'').
\end{defn}

Observe that in Definition~\ref{finwtdef} we stipulate that no bud of
a finite cubic lamination is the common image of both critical leaves.
This is because a cubic lamination with a wandering triangle must have
two critical leaves with disjoint orbits \cite[Theorem A]{blolev99}.

\subsection{Continuing to a cubic WT-lamination}

We will build larger and larger finite laminations, only achieving a
true lamination with a wandering triangle in the limit.
We say that a dynamical lamination $(\La', g')$ \emph{continues}
a dynamical lamination $(\La, g)$ if

\begin{enumerate}
\item $\La \subset \La'$ and $\bLa \subset \bLa'$;
\item $g$ and $g'$ have the same degree $3$; and
\item $g'|_{\dom(g)} = g$.
\end{enumerate}

A
natural way to continue a finite lamination uses admissible extensions
of the map $g$ to the circle.
A covering map $F:\uc \to \uc$ of degree three such
that $F|_{\dom(g)}=g$ is said to be an \emph{admissible} extension of
$g$ if $F$ restricted to any complementary interval of $\dom(g)=\bLa$
is strictly monotonically increasing.
Let $F$ be an admissible extension of $g$.
A \emph{forward continuation} is a continuation such that all elements
of $\La' \setminus \La$ are forward images of $\La$-classes. If $(\La,
g)$ is a finite cubic WT-lamination, then $F$
is a \emph{first forward-continuing extension} of $g$ if $(\La\cup
\{F(T^l)\}, F|_{\bLa \cup \{T^l\}})=(\La_F, F|_{\bLa \cup \{T^l\}})$
is a dynamical lamination (where $\T^l$ is the last triple of $\La$);
then $(\La_F, F|_{\bLa \cup \{T^l\}})$ is called the \emph{(first)
  forward continuation of $(\La, g)$ associated to $F$}.

A \emph{backward continuation} is a continuation such that all
elements of $\La' \setminus \La$ eventually map to $\La$-classes.  A
natural way to construct backward continuations is to pull classes
back under an admissible extension of $g$.  Let $F$ be an admissible
extension of $g$, let $U$ be a $\bLa$-unlinked class containing no
$\La$-class, and let $\mx$ be any $\La$-class contained in $F(U)$.
The \emph{unambiguous backward continuation} of $(\La, g)$ (pulling
$\mx$ into $U$ under $F$) is the backward continuation $(\La', g')$
such that $\La' = \La \cup \{\mx'\}$, where $\mx'$ is the first
$F$-preimage of $\mx$ in $U$, and \[g'(t) =
\begin{cases}
 g(t) & \mbox{if }t \in \bigcup \bLa\\
 F(t) & \mbox{if }t \in \mx'
\end{cases}.\]
As we prove in the following Lemma, unambiguous backward continuations
are combinatorially unique, and therefore may be accomplished without
reference to a particular admissible extension.  For this reason, we
can also iteratively construct unambiguous backward continuations
without worry.  We will also call such continuations unambiguous
backward continuations.

\begin{lem}\label{l:unambi}
  Let $(\La, g)$ be a critical lamination, let $U$ be a
  $\bLa$-unlinked class, and let $\mx$ be a $\La$-class.  If $F$ is an
  admissible extension of $g$ such that $F(U)$ contains $\mx$, then
  the unambiguous extension $(\La', g')$ pulling $\mx$ into $U$ under
  $F$ is defined and is a critical lamination of the same degree.
  Further, for any other admissible extensions $F'$ of $g$, the
  backward continuation pulling $\mx$ into $U$ under $F'$ is defined
  and isomorphic to $(\La', g')$.
\end{lem}

\begin{proof}
  Since $\La$ contains two critical leaves and is of degree $3$, then
  $F|_{U}$ is one-to-one and order preserving, and the image set
  $F(U)$ is independent of the particular extension $F$.  Construct
  $(\La', g')$ as the unambiguous backward continuation pulling $\mx$
  into $U$ under $F$.  Since $F|_U$ is order-preserving, the resulting
  lamination $(\La', g')$ is of the same degree.  Further, for any
  other extension $F'$, we see that the corresponding unambiguous
  continuation $(\La'', g'')$ is conjugate to $(\La', g')$ by the map
  $f:\bigcup \La' \to \bigcup \La''$ defined by
  \[f(t) =
  \begin{cases}
    t & \mbox{if } t \notin \mx'\\
    (F'|_U^{-1} \circ F)(t) & \mbox{if } t \in \mx'
  \end{cases}\] where $\mx'$ is the $g'$-preimage of $\mx$.
\end{proof}

Observe that a backward continuation is not always
unambiguous; given a class, if we wish to find a preimage in a
$\La$-unlinked class $U$ containing a last class of $\La$, different
admissible extensions can give rise to finite laminations that are not
order isomorphic.  In such a situation, we must specify exactly where
the pull back of the image class is located in $U$ with respect to all
last classes contained in $U$.  See Figure~\ref{fig:ambiguous}.1.
\begin{figure}\label{fig:ambiguous}
  \centering
  \includegraphics{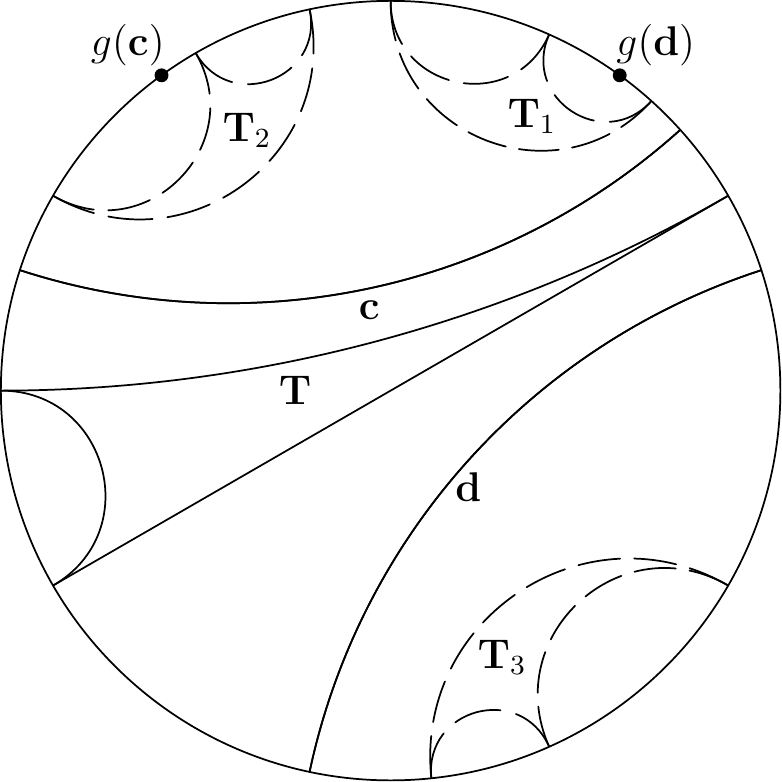}
  \caption{Consider the finite dynamical lamination $\La = \{\mc, \md,
    \T, g(\mc), g(\md)\}$ pictured here and a few possible backward
    continuations of $(\La, g)$ whose new elements are pictured in dotted
    lines. It is evident that all backward continuations of $(\La,g)$
    adding only a first preimage of $\T$ under $\md$ (like $\T_3$, as illustrated) are order isomorphic.
    However, the choice of a first preimage of $\T$ in the
    component of $\uc\sm \mc$, containing the last classes $g(\mc)$ and $g(\md)$,
    is ambiguous; $\T_1$ and $\T_2$ illustrate choices giving rise to
    backward continuations which are not order isomorphic.}
\end{figure}

\begin{defn}\label{d:itine}
  For a critical lamination $(\La, g)$ with critical leaves $\C=\{\mc,
  \md\}$, a \emph{$\La$-itinerary} is a finite string $i_0i_1\ldots i_n$ of
  $\C$-unlinked classes and critical classes.  A $\La$-itinerary which
  contains no critical classes is called \emph{non-critical}.  A
  $\La$-itinerary whose last element, but no other, is a critical leaf
  is called \emph{end-critical}. We associate to $x\in \uc$ (not
  necessarily contained in any $\La$-class) the \emph{maximal $\La$-itinerary}
  $\bi(x)=I_0I_1I_2\dots$ of $x$, which is the (typically finite)
  maximal well-defined sequence of $\C$-unlinked classes and critical classes
  such that for every admissible extension $F$ of $g$ we have that
  $F^j(x)\in I_j$; $\bi(x)$ begins at the moment zero, i.e.
  at the $\C$-unlinked class or critical class containing $x$.
  An initial string of the maximal
  $\La$-itinerary of $x$ is said to be a \emph{$\La$-itinerary} of $x$.
\end{defn}

An arbitrary itinerary is not assumed to be realized by an element
of a critical lamination, but is considered only as a
\emph{potential} itinerary of an element of $\La$. However, if $\mm$
is a $\La$-class we can also define for it an itinerary of different
length. Namely, the \emph{true $\La$-itinerary} $\ti(\mm)$ is
the maximal sequence of $\C$-unlinked classes and critical classes
containing the $g$-images of $\mm$ as long as these $g$-images are
still $\La$-classes. Since the $g$-images of $\La$-classes are
$\La$-classes, $\ti(\mm)$ coincides with an appropriate initial
string of $\bi(x)$ for each $x\in \mm$.
For example, the true itinerary of a last
class has only one entry, while its maximal itinerary may have many
entries. If $\La$ is fixed, we may talk about itineraries without
explicitly mentioning $\La$.


\subsection{Main continuation lemma}

Here, we describe the main ingredients used in the construction of
finite laminations in the next section. First we define the
following concept. Suppose that $A$ and $B$ are either edges or classes of a finite
lamination $\La$ such that $A\cap B$ is at most a point.
Consider components of $\ol{\disk}\sm (\ch(A)\cup \ch(B))$ and choose among
them the unique component $X$ which borders both $A$ and $B$. This
component $X$ is called the \emph{part of\, $\ol{\disk}$ between $A$
and $B$}. Also, recall that by \emph{leaves} we mean all
$\La$-classes consisting of two points (as well as their convex
hulls) while by \emph{edges} we mean boundary chords of convex hulls
of $\La$-classes.

\begin{figure}\label{fig:geo_illus}
  \centering
  \includegraphics[width=0.6 \linewidth]{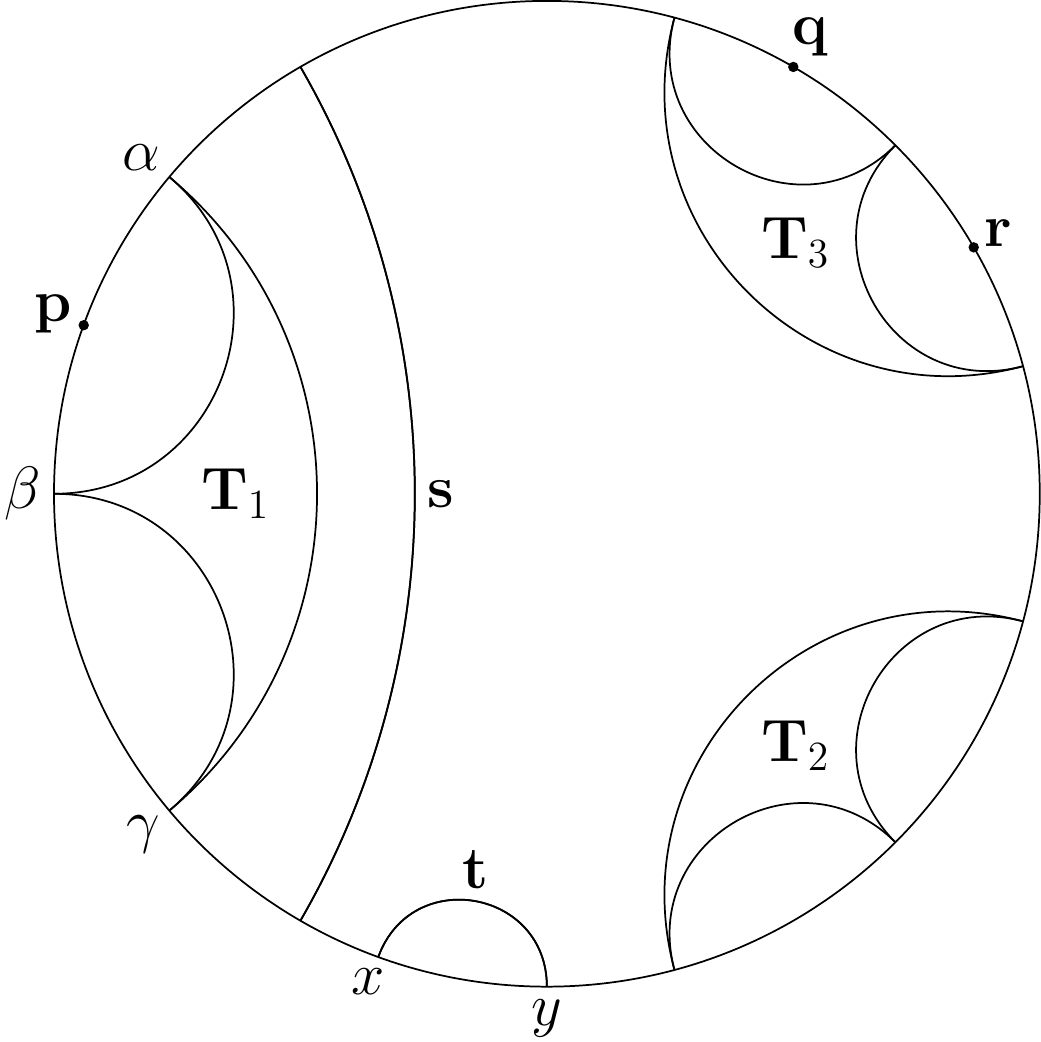}
  \caption{An illustration of concepts from Definition~\ref{adj}.  The
    finite lamination in question is $\La = \{\mp, \mq, \mr, \ms, \mt,
    \T_1, \T_2, \T_3\}$.  The triple $\T_1$ is leaf-like, and is
    adjacent to both the bud $\mp$ and the leaf $\ms$.  Its long edge
    is therefore $\{\alpha, \gamma\}$, its short edge is $\{\alpha,
    \beta\}$, and its empty edge is $\{\beta, \gamma\}$.  The class
    $\T_3$ is  adjacent to both $\mq$ and $\mr$.  The
    triple $\T_2$ is not leaf-like, and is not adjacent to any other
    class.  The arc under the leaf $\mt$ is $[x,y]$. There are no arcs
    under $\ms$. The arcs under the different edges of $T_1$ are
    $[\alpha, \beta]$, $[\beta, \gamma]$, and $[\gamma, \alpha]$.}
\end{figure}

\begin{defn}[Adjacent, leaf-like]\label{adj}
  Two classes (or edges) are \textbf{adjacent} if the part of
  $\ol{\disk}$ between them is disjoint from $\La$. A triple $\T$ is
  \textbf{leaf-like} if exactly one of the components of $\uc \sm \T$ is
  disjoint from $\La$; if $\T$ is leaf-like  and adjacent to a bud $\mx$,
  then the \textbf{short edge} of $\T$ is the edge $\ms$ of $\T$
  adjacent to $\mx$, the \textbf{long edge} of $\T$ is the edge $\mell$ of $\T$
  separating $\T \sm \mell$ (and $\mx$) from the rest of $\La$, and the
  \textbf{empty edge} of $\T$ is the remaining edge of $\T$. A leaf is
  always considered to be leaf-like. The \textbf{arc
    under a leaf} $\mt$ is the (open) component of $\uc \setminus \mt$
  which contains no class of $\La$, if such an interval exists.
  Finally, if $\ol{a}$ is an edge of a triple $\T$, then the
  \textbf{arc under} $\ol{a}$ is the component of $\uc \setminus
  \ol{a}$ which does not contain $\T \setminus \ol{a}$. Observe that
  the notion of the arc under a leaf and that of the arc under an edge
  of a triangle have somewhat different meaning. These notions are
  illustrated in Figure~\ref{fig:geo_illus}.1.
\end{defn}

\begin{defn}[$\mc$-lamination]\label{def:c-moment}
  Let $(\La, g)$ be a finite cubic WT-lamination. We say that $(\La, g)$ is
  a \textbf{$\mc$-lamination} if there is a triple $\T^\md$, a
  preimage $\mm(\lam)=\mm$ of $\mc$, and a preimage
  $\mell(\lam)=\mell$ of $\md$ such that the following hold.

  \begin{enumerate}
  \item $\T^\md$ is adjacent to both $\md'$ and $\mm$.
  \item \label{item:follow} If $r > 0$ is such that
    $g^r(\mm)=\mc'$, then the last triangle $\T^l$ is equal to $g^r(\T^\md)$.
    Further, if $0 \le k \le r$, then $g^k(\T^\md)$ is leaf-like and
    adjacent to $g^k(\mm)$.  In particular, the last triple is
    leaf-like and adjacent to $\mc'$.
  \item If $0 < k \le r$, the short edge of $\T^\md$ facing $\md'$ maps
    under $g^k$ to an empty edge of $g^k(\T^\md)$.
  \item The long edge of the last triple is adjacent to $\mell$.
  \end{enumerate}
  By replacing $\mc$ with $\md$ above, we obtain the definition of a
  \textbf{$\md$-lamination $\lam$}. More precisely, we say that $(\La, g)$ is
  a \textbf{$\md$-lamination} if there is a triple $\T^\mc$, a
  preimage $\mm(\lam)=\mm$ of $\md$, and a preimage
  $\mell(\lam)=\mell$ of $\mc$ such that the following holds.

  \begin{enumerate}
  \item $\T^\mc$ is adjacent to both $\mc'$ and $\mm$.
  \item \label{item:followd} If $r > 0$ is such that
    $g^r(\mm)=\md'$, then the last triangle $\T^l$ is equal to $g^r(\T^\mc)$.
    Further, if $0 \le k \le r$, then $g^k(\T^\mc)$ is leaf-like and
    adjacent to $g^k(\mm)$.  In particular, the last triple is
    leaf-like and adjacent to $\md'$.
  \item If $0 < k \le r$, the short edge of $\T^\mc$ facing $\mc'$ maps
    under $g^k$ to an empty edge of $g^k(\T^\mc)$.
  \item The long edge of the last triple is adjacent to $\mell$.
  \end{enumerate}
\end{defn}

\begin{figure}
  \centering
  \includegraphics{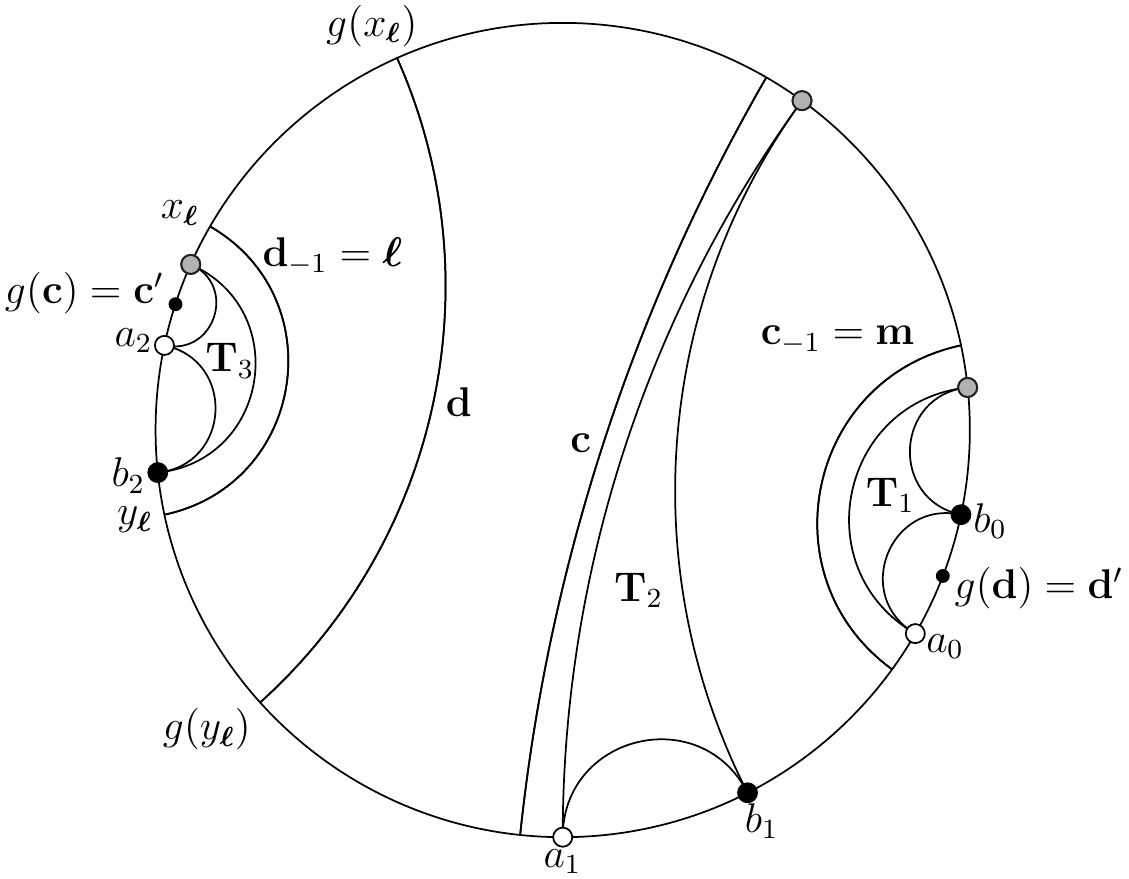}
  \caption{A simple $\mc$-lamination.  Like-shaded vertices are in
    the same orbit.  The monotone sequence of intervals is $([a_0,
    b_0], [a_1, b_1], [a_2, b_2], [g(x_\ell), g(y_\ell)])$, $g^2(\mm)=\mc'$
    and $r=2$.  The monotone
    itinerary is $(C, C, D, D)$, where $C$ denotes the $\{\mc,
    \md\}$-unlinked class whose boundary is $\mc$, and likewise $D$
    for $\md$.}
\end{figure}

Given a $\mc$-lamination $(\La,g)$, we observe that a particular
sequence of intervals describes the maximal $\La$-itinerary of $\md'$
in the sense of Definition~\ref{d:itine}. Let $\{a_0, b_0\}$ denote
the endpoints of the short edge of $\T^\md$, so $\md' \in K_0 = [a_0,
b_0]$. As in Definition~\ref{def:c-moment}, let $r$ denote the integer
such that $g^r(\T^\md) = \T^l$.  We then set $a_k = g^k(a_0)$ and $b_k
= g^k(b_0)$ for each $k \le r$.  Since $\{a_k, b_k\}$ is the 
empty edge of $g^k(\T^\md)$, it is apparent from the definition of a
$\mc$-lamination that $G^k([a_0,b_0]) = [a_k,b_k]=K_k$ for any
admissible extension $G$ of $g$ and any $k \le r$.

Let $\{x_\mell, y_\mell\}$ be the endpoints of $\mell$ labeled so
that  $K_r \subset [x_\mell, y_\mell]$.  Then $G^r(\md') \in
K_r \subset [x_\mell,y_\mell]$ for any admissible extension $G$ of
$g$. In other words, the first $r$ entries of the maximal
$\La$-itinerary of $\md'$ are uniquely determined by the images of
the short edge of $\T^\md$ facing $\md'$, i.e., empty edges of
triples $g^i(\T^\md)$ (see Definition~\ref{def:c-moment}(4)).
 Then, from the moment $r$ on, the
images of $\md'$ cannot be described through images of edges of
$\T^l$ because $\T^l$ is the last triple and images of $\T^l$ are
not defined.

However images of $\mell$ are defined and force the next uniquely
determined string of the maximal $\La$-itinerary of $\md'$. More
precisely, we are guaranteed that $G^{r+k+1}(\md') \in
[g^{k+1}(x_\mell), g^{k+1}(y_\mell)]$ if the interval $[g^k(x_\mell),
g^k(y_\mell)]$ contains no critical leaf.  We define $K_{r+1} =
[g(x_\mell), g(y_\mell)]$.  If $K_{r+k} = [g^k(x_\mell),
g^k(y_\mell)]$ is defined and contains no critical leaf, then we set
$K_{r+k+1} = [g^{k+1}(x_\mell), g^{k+1}(y_\mell)]$.  Since $\mell$ is
eventually critical, for some $N$ we have that $K_N$ contains a
critical leaf, and the process stops.  If $G$ is any admissible
extension of $g$, then $N$ is maximal such that $G^N|_{K_0}$ is
one-to-one. Set $N(\lam)=N$.

Recall that $\C=\{\mc, \md\}$. It follows that if $K_N$ is not the
closure of a connected $\C$-unlinked class then $i_0\dots i_{N-1}$ is
the maximal $\La$-itinerary of $\md'$ (and of every point in the
interior of $K_0$). Choose $i_N$ to be a \emph{connected $\C$-unlinked class
contained in $K_N$}. If $K_N$ is the closure of a connected
$\C$-unlinked class, then this class must be $i_N$ and the maximal
$\La$-itinerary of $\md'$ is $i_0\dots i_{N-1}i_N$.

The above introduced notation (number $r$, intervals $K_i=[a_i, b_i]$,
$\mell=\ol{x_\mell, y_\mell}$, etc.) will be used throughout the paper
when we talk about $\mc$-laminations.

\begin{defn}\label{d:seqint}
  For a $\mc$-lamination $(\La,g)$, call $K_0, \ldots, K_{N(\lam)}$ the
  \textbf{monotone sequence of intervals}.  (As usual, the
  parenthetically indicated dependence of the intervals or $N$ upon
  $\La$ may be omitted.)  The itinerary $i_0^\mon\dots i_N^\mon)=\monitin(\La)$ is
  called the \textbf{monotone itinerary (of $(\La, g)$)}. Define a
  monotone sequence of intervals and a monotone itinerary of a
  $\md$-lamination similarly.
\end{defn}

\begin{rem}\label{rem:class_itins}
  In this remark we use the notation from
  Definition~\ref{def:c-moment} (in particular, $r$ below is the positive integer
  such that $g^r(\mm)=\mc'$ and $r\le N$). We claim that the true itinerary of
  $\mm$ does not have $i_0^\mon\dots i_N^\mon = \monitin(\La)$ as its initial
  segment, because the orbit of $\mm$ passes through $\mc$ in the
  first $r$ steps, and $\monitin(\La)$ is non-critical. Let us also show
  that an initial segment of $\ti(\mell)$ cannot coincide with
  $i_r^\mon\dots i_N^\mon$. This is because either $\ol{i_N^\mon} \subsetneq K_N$, in which case the endpoints of $K_n$ (i.e., the endpoints of $g^{N-r}(\mell)$) cannot lie within $K_N$; or $g^{N-r}(\mell) = \md$, while $\monitin(\La)$ is not end-critical by definition.
\end{rem}

The following is another  consequence of the definition.

\begin{lem}\label{lem:charitin}Let $(\La,g)$ be a $\mc$-lamination.
  Suppose $(\La', g')$ is a continuation of $(\La, g)$ such that
  $\T^l(\La)$ is still adjacent to $\mell(\La)$ and to $\mc'$ in
  $\La'$.  If the true $\La'$-itinerary of a non-degenerate $\La'$-class
  $\mz$ begins with $\monitin(\La)$, then $\mz$ lies in $K_0$.
\end{lem}
\begin{proof}
  We will show by induction on $t$ that $(g')^{N-t}(\mz) \subset
  K_{N-t}$.  Since $g'$ is a continuation of $g$, we may use $g$ and
  $g'$ interchangeably on $\La$-classes.

  That $(g')^N(\mz) \subset i_N^\mon \subset K_N$ follows by definition.
  Suppose by induction that $(g')^{j+1}(\mz) \subset K_{j+1}$ for $0
  \le j < N$.  We see that $(g')^j(\mz)$ is the unique class in $i_j^\mon$
  which maps onto $(g')^{j+1}(\mz)$.  If $j \neq r$, then $(g')^j(\mz)
  \subset K_j$ since $g'|_{i_j^\mon}$ is order-preserving.  On the other
  hand, if $j=r$, then $K_{j+1} = [g(x_\mell), g(y_\mell)]$ and
  $(g')^j(\mz) \subset [x_\mell, y_\mell]$ since $g'|_{i_r}$ is
  order-preserving.  However, we assume that $\mell$ is adjacent to
  $\T^l$ and $\mc'$ in $\La'$, and $(g')^j(\mz)$ is neither $\mell$
  nor $\mc'$ by Remark~\ref{rem:class_itins}, so $(g')^j(\mz) \subset
  [a_r, b_r] = K_r$.  Hence, by induction, $\mz \subset K_0$.
\end{proof}

Lemma~\ref{lem:extend1_lam} is an important inductive step in the
construction.

\begin{lem}\label{lem:extend1_lam}
  Let $(\La, g)$ be a $\mc$-lamination, and let $\tau = j_0 \ldots j_M
  \mc$ be an end-critical $\La$-itinerary where $M \le N(\La) -2$.
  Then there is a backward continuation $(\La', g')$ containing a leaf
  $\mell(\La')$ such that the following claims hold.
  \begin{enumerate}
  \item The leaf $\mell(\La')$ is a $g'$-preimage of $\mc$, and every
    leaf in $\La' \setminus \La$ is a forward image of $\mell(\La')$.
  \item The leaf $\mell(\La')$ is adjacent to $\md'$ and $\T^\md$.
  \item The true itinerary of $\mell(\La')$
    begins with $(i_0^\mon, \ldots, i_N^\mon)$, ends with $\tau$, and these two segments of the itinerary are disjoint.
  \item The common last triple $T^l(\La) = T^l(\La')$ of both
    laminations is still adjacent in $\La'$ to both $\mell(\La)$ and
    $\mc'$.
  \item The length of the maximal $\La'$-itinerary of $\md'$ is of length
  at least $N(\La)+2$.
  \end{enumerate}
\end{lem}

\begin{proof}
  We will construct $(\La', g')$ by adding a sequence of consecutive
  pullbacks of $\mc$ until we come to the leaf $\mell(\La')$. We will
  use the following rule when dealing with ambiguous pullbacks: any
  pullback into $(x_\mell, y_\mell)$ must be under the empty edge of
  $\T^l(\La)$.  (Any pullbacks into $K_0$ would also be
  ambiguous, but such will not occur until the last pullback we take
  in this proof.)

  Take iterative preimages of $\mc$ along the itinerary $\tau$, using
  the rules above for ambiguous pullbacks.  Note that no pullback is
  taken into the interval $K_0$, since $\tau$ is too short to contain
  the itinerary of any point in $K_0$.  Denote the leaf so obtained by
  $\mc^\tau$.  (Note that part or all of the orbit of $\mc^\tau$ may
  already be represented in $\La$; we do not add extra copies in order
  to preserve that $(\La', g')$ is a degree 3 finite lamination.)

  Next we consider the problem of finding a preimage of $\mc_\tau$
  in $K_0$ which increases the length of the maximal itinerary of
  $\md'$. First let us show that $i_0=\dots=i_N$ is impossible.
  Recall from the discussion before
  Definition~\ref{d:seqint} that $i_N$ is always a \emph{connected
  $\C$-unlinked class contained in $K_N$}. By the definition of the
  monotone sequence of intervals, $K_{N-1}\subset i_N\subset K_N$.
  Since the map $g$ preserves orientation on $i_N$ and $i_N$ is an arc,
  it follows that $K_0\subset K_1\subset \dots \subset
  K_{N-1}\subset K_N$. Thus, $K_1=g(K_0)$ contains $\md'$ and is not the interval under an empty edge
  of $g(\T^{\md})$, a contradiction with the definition of a
  $\mc$-lamination.

  Thus, we may assume that some of the sets $i_0^\mon, \dots, i_N^\mon$
  are distinct. Choose a $\C$-unlinked class $C \neq i_{N-1}^\mon$ which
  does not contain the common image of the endpoints of $i_N^\mon$, and
  choose a $\C$-unlinked class $Z$ which is neither $i_0^\mon$ nor
  $i_{N-1}^\mon$.  Then, if $k > 1$ is such that 
  $\monitin(\La)$ contains no string of $Z$'s of length $k$ (denoted
  $Z^k$), we choose the itinerary $\tau' = i_0 \ldots i_N C Z^{k}
  \tau$.  Let us show that $\tau'$ contains exactly one copy of the
  string $i^\mon_0 \ldots i^\mon_{N-1}$. Indeed, $i^\mon_1\dots i^\mon_N\ne i^\mon_0 \ldots
  i^\mon_{N-1}$ because not all entries $i^\mon_0, \dots, i^\mon_N$ are equal. Also,
  $i^\mon_2\dots i^\mon_NC\ne i^\mon_0 \ldots i^\mon_{N-1}$ because $C\ne i^\mon_{N-1}$. No
  substring in $\tau'$ which ends with $Z$ can be a copy of $i^\mon_0\ldots
  i^\mon_{N-1}$ because $Z\ne i^\mon_{N-1}$. No copy of $i^\mon_0 \ldots i^\mon_{N-1}$ can
  contain a string of $k$ $Z$'s by the choice of $k$. Finally, no copy
  of $i^\mon_0 \ldots i^\mon_{N-1}$ can begin with $Z$, since $Z \neq i^\mon_0$. So
  indeed, $\tau'$ contains exactly one copy of the string $i^\mon_0 \ldots
  i^\mon_{N-1}$.

  Now, construct preimages of $\mc_\tau$ moving back from $\mc_\tau$
  along the itinerary $i^\mon_0 \ldots i^\mon_N C Z^k$.  By the previous
  paragraph, all such pullbacks are taken outside of $K_0$ except the
  last.  The last pullback $\mell(\La')$ is to be taken in $K_0$ by
  Lemma~\ref{lem:charitin}, and we choose it to be adjacent to
  $\md'$. Call the resulting lamination $(\La', g')$.  Now every point
  of the component of $\uc\sm \mell(\La')$ containing $\md'$ has
  maximal $\La'$-itinerary beginning with $i^\mon_0 \ldots i^\mon_N$ and has length at least $N(\La)+2$.
  This completes the proof.
\end{proof}

Given a $\mc$-lamination or a $\md$-lamination $\La$, denote by
$\lo(\La)=\lo$ the long edge of $\T^l=\T^l(\La)$ and by $\sh(\La)=\sh$ its short edge.

\begin{lem}\label{lem:extend_lam}
  Let $(\La, g)$ be a $\mc$-lamination, and let $\tau = j_0 j_1 \ldots
  j_k \mc$ be an end-critical $\La$-itinerary for some $k \le
  N(\La)-2$ (see Definition~\ref{d:seqint}). Let $t_0 t_1 \dots t_m
  \md$ be an initial segment of the true $\La$-itinerary of
  $\ell(\La)$.  Then there is a $\md$-lamination $(\La'', g'')$ which
  continues $(\La, g)$ in two steps: {\rm(1)} construct $(\La', g')$
  by $\tau$ as in Lemma~\ref{lem:extend1_lam}, and then {\rm(2)} add,
  in the appropriate way, a segment of the orbit of $\T^l(\La)$ which
  follows the orbit of $\ell(\La)$ so that an initial segment of the
  true $\La''$-itinerary of $\T^l(\La)$ coincides with $t_0\dots
  t_m$. Moreover, the maximal $\La''$-itinerary of $\md'$ is of length
  at least $N(\La)+2$, and $\ell(\La'')=\ell(\La')$ has  $\tau$ as a subsegment of its
  true $\La''$-itinerary.
\end{lem}

The statement of Lemma~\ref{lem:extend_lam} also holds, with appropriate swapping of
$\mc$ and $\md$ related objects.

\begin{proof}
  Let $(\La', g')$ be the backward continuation given by
  Lemma~\ref{lem:extend1_lam}.  We shall continue $(\La', g')$ to a
  $\md$-lamination which will be constructed by adding images of
  $\T^l(\La')$ to $(\La', g')$ to obtain a $\md$-lamination $(\La'',
  g'')$ with $\mell(\La')=\mell(\La'')$. Recall that since $\La$ is a
  $\mc$-lamination, there is also a leaf $\mell(\La)$ which is
  different from $\mell(\La')$. Let $\T^l=\T_n$.

  Let $s$ and $t$ be such that $g^s(\mell(\La))=\md$ and
  $g^t(\md)=\md'$.  Since $\lo$ is adjacent to $\mell(\La)$ and $\sh$
  is adjacent to the last class $\mc'$, we can define $\T_{n+1}$,
  \ldots, $\T_{n+s+t}$ up to isomorphism so that for any $i \in \{0,
  \ldots, s+t\}$ we have
  \begin{enumerate}
  \item $(g'')^i(\lo)$ is adjacent to $g^i(\mell(\La))$, and
  \item one component of $\uc \setminus (g'')^i(\sh)$ contains no
    classes of $\La$.
  \end{enumerate}
  Note that by the properties of $\mc$-laminations all leaves
  $g^i(\mell(\La))$ are located outside $[x_\mell, y_\mell]$. Hence
  the construction implies that $\T_{n+i}$ is outside of $[x_\mell,
  y_\mell]$ for all $0<i\le s+t$, which implies that the short edge
  $\sh$ of $\T_n$ is still adjacent to $\mc'$ in $\La''$.  It is now
  easy to verify that $(\La'', g'')$ is a $\md$-lamination and has all
  the desired properties.
\end{proof}

Beginning with any $\mc$-lamination $(\La_0, g_0)$, by
Lemma~\ref{lem:extend_lam} one can construct an increasing sequence of
critical laminations $\left( (\La_i, g_i)\right)_{i=1}^\infty$, where
$(\La_{i+1}, g_{i+1})$ continues
$(\La_i,g_i)$. This is done on the basis of a sequence of end-critical
finite itineraries used in the applications of
Lemma~\ref{lem:extend_lam} as a sequence $\tau$.  As we rely on
Lemma~\ref{lem:extend_lam}, we continue a $\mc$-lamination to a
$\md$-lamination, then to a $\mc$-lamination, etc. Each critical
lamination contains a triple $\T_1$, and the orbit of $\T_1$ is
continued in each successive critical lamination by adding a forward
segment to the orbit of $\T^l(\La)$ which follows the orbit of
$\ell(\La)$ and has the appropriate initial segment of its true
itinerary.

In other words, on each step, when we continue $\La$, we fulfill two
tasks: (1) create the leaf $\ell(\La')=\ell(\La'')$ which has any
given subsegment in its true itinerary, and (2) add a segment of the
orbit of $\T^l(\La)$ which has a subsegment coinciding with the
appropriate (up to the critical leaf) subsegment of the true itinerary
of $\ell(\La)$.  To go on with this construction we need to choose
finite itineraries which determine the construction on each step.  We
can choose them to cover all possible itineraries. Then applying the
construction infinitely many times we obtain the limit lamination
$\bigcup_{i=1}^\infty \La_i$ equipped with the limit map $g$,
containing (combinatorially) a wandering triangle which has an
itinerary in which any finite itinerary shows at least once.

\section{Realization and Density}
\subsection{Realizing the combinatorial lamination}\label{constr}

In this section, we show that any $\mc$-lamination $(\La_0, g_0)$ can
be continued (up to conjugacy by an order isomorphism) to a
$\si_3$-invariant lamination $\La$ containing a wandering triangle
whose forward orbit is dense in the lamination.  We first construct a
combinatorial lamination $(\La_\infty, g_\infty)$ with appropriate
properties by repeated application of Lemma~\ref{lem:extend_lam}.  We
show that $g_\infty$ restricted to the forward orbit of $\T_1$
satisfies an expansion property (see Definition~\ref{defn:top_exact}),
and is therefore conjugate to a restriction of $\sigma_3$ to some
subset $A$.  The $3$-invariant lamination $\La$ will be the closure of
the induced lamination on $A$.

Recall that end-critical itineraries were defined in
Definition~\ref{d:itine}.

\begin{defn}
  A sequence $(\tau_n)_{n=0}^\infty$ of end-critical
  itineraries ending in $\mc$ is called \emph{full} if the length of $\tau_n$
  converges to infinity and for any finite
  precritical itinerary $\tau$ ending in $\mc$ there is an integer $n$
  such that the itinerary $\tau_n$ ends in $\tau$. Similarly we define
  a full sequence of pairwise distinct precritical itineraries ending
  in $\md$.
\end{defn}

Let $(\La_0, g_0)$ be fixed for the rest of the section; we assume
that the lengths of the itineraries of $\mc$ and $\md$ are both
greater than two.  Choose a full sequence of itineraries
$(\tau_n^\mc)_{n=0}^\infty$ ending in $\mc$ such that the following hold.
\begin{itemize}
\item The length of $\tau_0^\mc$ is less than the length
  of the $\La_0$-itinerary of $\md'$ by at least $2$, and
\item The length of $\tau_{i+1}^\mc$ is at most one more than the
  length of $\tau_i^\mc$.
\end{itemize}
Similarly define $(\tau_n^\md)_{n=1}^\infty$.  We then inductively
define laminations $((\La_i, g_i))_{i=1}^\infty$, where even indices
correspond to $\mc$-laminations and odd indices correspond to
$\md$-laminations, as follows.
\begin{itemize}
\item If $\La_i$ is defined for an even integer $i = 2k$, we
  use Lemma~\ref{lem:extend_lam} with $\tau^\mc_k$ to obtain a
  $\md$-lamination $(\La_{i+1}, g_{i+1})$.
\item If $\La_i$ is defined for an odd integer $i = 2k-1$,
  we use Lemma~\ref{lem:extend_lam} with $\tau^\md_k$ to obtain a
  $\mc$-lamination $(\La_{i+1}, g_{i+1})$.
\end{itemize}
Then $\bigcup \La_i = \La_\infty$ with the natural map $g_\infty$ is
an infinite combinatorial lamination with a wandering class $\T_1$ and
a full set of preimages of both $\mc$ and $\md$, i.e., exactly one
leaf corresponding to each end-critical $\La_0$-itinerary.

Due to the properties of $\mc$- and $\md$-laminations, the orbit of
$\T_1$ in $(\La_\infty, g_\infty)$ is most easily analyzed in terms of
``closest approaches''.  For example, $\T^l(\La_{2k})$ is adjacent to
$\mc'$ in $\La_{2k}$, so this constitutes a closest approach in the
lamination $\La_{2k}$.  These sequences of closest
approaches are particularly important to us, so we will keep two
sequences in mind: let $(p_n)_{n=0}^\infty$ be the sequence so that
$\T_{p_n} = \T^l(\La_n)$, and set $(k_n)_{n=0}^\infty$ to be the
sequence so that $\T_{k_n}$ is adjacent in $\La_n$ to the appropriate critical
leaf, following its orbit in a leaf-like manner to $\mc'$ or $\md'$.
Hence, we have $k_1 < p_1 < k_2 < p_2 < \ldots$.  Further $p_1 -k_1 = p_3 -
k_3 = \cdots = p_{2n+1} - k_{2n+1} = \cdots$, and $p_2-k_2 = p_4-k_4 = \ldots =
p_{2n}-k_{2n} = \cdots$ since these are the lengths of the segments of
orbits from $\mc$ to $\mc'$, and $\md$ to $\md'$, respectively.

Recall that an edge (termed the \emph{short edge}) of $\T_{k_n+1}$
is adjacent in $\La_n$ to the image of a critical leaf, so no
previous triple lies in the interval underneath that edge.  In
continuing its orbit, the images of this edge are adjacent to the
images of the critical leaf, so no previous triple lies under the
corresponding edges of $\T_{k_n+2}$, $\T_{k_n+3}$, \ldots,
$\T_{p_n}$.  Similarly, the part of the disk between $\T_{p_n}$ and
$\ell(\La_n)$, as well as the part of the disk between their first
$p_{n+1}-p_n$ images, contains no previous triple. Observe that
$p_{n+1}-p_n\to \infty$, since $(p_{2n})_{n=1}^\infty$ and
$(p_{2n+1})_{n=1}^\infty$ are monotone and must eventually
accommodate all itineraries $\tau^\mc_i$ and $\tau^\md_i$.

We will work with the lamination $\La_\T = \{\T_n\,\mid\,n \ge 1\}$
with basis $A_\T = \bigcup_{n=1}^\infty \T_n$ and the map $g_\T =
g_\infty|_A$.  Note that $A_\T$ is forward-invariant under the map
$g_\T$.  We wish to show that $g_\T$ is conjugate to a restriction
of $\sigma_3$.  To proceed we need Theorem~\ref{extinv}; to state it
we need Definitions~\ref{defn:top_exact} and ~\ref{defn:extendable}.
Recall that maps of degree $3$ are defined right before
Definition~\ref{finwtdef}. Below we work with a set $A\subset \uc$
which is considered with the circular order on it; the topology of
$A$ as a subspace of $\uc$ is ignored.

\begin{defn}\label{defn:top_exact}
  A map $f:A\to A$ of degree $3$ is said to be \emph{topologically
    exact} if for each pair of distinct points $x,y \in A$ there
  exists an integer $n\ge 1$ such that either $f^n(x)=f^n(y)$, or
  $f([f^n(x), f^n(y)]_A)\not\subset [f^{n+1}(x), f^{n+1}(y)]_A$. If we
  want to emphasize the precise value of $n$, we say that $f$ maps
  $[x, y]_A$ \emph{out of order in $n+1$ steps}.
\end{defn}

\begin{defn}[$\si$-extendable]\label{defn:extendable}
  A degree $3$ map $f:A\to A,$ $A\subset \uc$ is
  \emph{$\si$-extendable} if for some $\si_3$-invariant set $A'\subset
  \uc$ the map $f|_A$ is conjugate to the map $\si_3|_{A'}:A'\to A'$
  and the conjugation is a circular order preserving bijection. Any
  such $A'$ is called an \emph{embedding} of $A$ into $\si_3$.
\end{defn}

We can now state an important technical result proven in \cite{bo04}.

\begin{thm}[{\cite[Theorem 3.3]{bo04}}]\label{extinv}
  If $f:A\to A\subset \uc$ is a topologically exact map of degree $3$,
  and $A$ is countable without fixed points, then $f$ is
  $\si$-extendable.
\end{thm}

\begin{lem}\label{lem:sigma_extendable}
  The map $g_\T$ is $\si$-extendable of degree $3$.
\end{lem}

\begin{proof}
  Let $x,y \in A_\T$ be distinct points.  We consider two main cases:
  either $[x,y]_{A_\T}$ contains the edge of a triple or
  $[x,y]_{A_\T}$ contains vertices from two distinct triples.

  First, suppose that $[x,y]_{A_\T}$ contains the edge $\{a,b\}$ of a
  triple $\T_k \in \La_i$.  The idea of the proof in this case is based upon
  the fact that, by construction, every edge of a triangle eventually becomes a long edge, at which point the interval underneath
  which will map out of order. To make this more precise, assume that
  $\T_k \in \La_i$ where $i$ is minimal and $[a,b] \subset
  [x,y]$.  Then $\T_k$ maps to $\T_n = \T^l(\La_i)$; suppose that
  $[x,y]_{A_\T}$ has not mapped out of order by this time.  Let $a' =
  (g')^{n-k}(a)$ and $b' = (g')^{n-k}(b)$. If $[a',b']_{A_\T}$ contains a
  critical leaf, then $[x,y]_{A_\T}$ maps out of order in $n-k+1$ steps.
  Otherwise, we have three cases.
  \begin{enumerate}
  \item Suppose $\{a',b'\}$ is the empty edge of $\T_n$, i.e., $[a',b']$
    contains no classes of $\La_i$, and the other point of $\T_n$ is
    not in $[a',b']$.  Considering $\T_n$ in $\La_{i+1}$, we see by
    Lemma~\ref{lem:extend_lam} that $\{a',b'\}$ maps to the long edge $\{a'',
    b''\}$ of $\T^l(\La_{i+1})$.  Since order is preserved, $[a'',
    b'']$ does not contain the other point of $\T^l(\La_{i+1})$, so
    $[a'', b'']$ contains both critical leaves.  Since $[a'', b'']_{A_\T}$ maps
    out of order in one step, $[x,y]_{A_\T}$ eventually maps out of order.
  \item Suppose $\{a',b'\}$ is the short edge of $\T_n$.  Then, by
    Lemma~\ref{lem:extend_lam} $\{a',b'\}$ maps to the empty edge of
    $\T^l(\La_{i+1})$.  By the previous case, we see that $[x,y]_{A_\T}$
    eventually maps out of order.
  \item Suppose that $\{a',b'\}$ is the long edge of $\T_n$.  Since
    $[a',b']$ does not contain the endpoints of a critical leaf,
    $[a',b']$ must contain the empty edge of $\T_n$.  We have
    therefore already shown that $[a', b']_A$, and therefore
    $[x,y]_{A_\T}$, eventually maps out of order.
  \end{enumerate}

  Suppose now that $x$ and $y$ lie in different triangles $\T_{m_1}$
  and $\T_{m_2}$ with $m_1 < m_2$. In this case the idea of the proof is based upon
  the fact that the arc $[x,y]$ will eventually cover an endpoint of a critical leaf,
  then its last critical value which, by construction,
   implies that it will cover an edge of a triangle and by the
  previous case will map out of order. The formal argument follows.
  Assume by way of contradiction
  that $[x,y]_{A_\T}$ never maps out of order. Let $\La_i$ be the first
  lamination in which $\T_{m_2}$ appears. Then $\T_{m_2}$ eventually
  maps to the last triple $\T_{m_2+k}$ of $\La_i$ which is leaf-like
  adjacent to a last image of a critical leaf (say $\mc'$). Then,
  since $\La_i$ is a $\mc$-lamination, there is a, say, $n$-th
  pullback of $\md$ which is leaf-like adjacent to the long (in the
  sense of $\La_i$) edge of $\T_{m_2+k}$.  This pullback separates
  $\T_{m_2+k}$ from all triples in $\La_i$, including $\T_{m_1+k}$, so
  the interval $[g^k(x),g^k(y)]$ contains one of its endpoints. By the
  assumption we get that $[g^{k+n+1}(x), g^{k+n+1}(y)]$ contains the
  critical value $g(\mc)$ in its interior. By construction, we see
  that $[g^{k+n+1}(x), g^{k+n+1}(y)]$ contains infinitely many
  triangles and hence by the previous paragraph a higher power of $g$
  will map $[x, y]_{A_\T}$ out of order.

  Since each interval between points of $A_\T$ eventually maps out of
  order, $g$ is topologically exact.  This proves that $g$ is
  $\si$-extendable by Theorem~\ref{extinv}.
\end{proof}

Hence, $g_\T$ is conjugate to a restriction $\si_3|_A$ via circular
order isomorphism $h:A_\T \to A$ for some $A \subset \uc$.  The map
$h$ is not unique; from now on we fix it. Set $\hat \T_1 =
h(\T_1)$. Let us study properties of the $\si_3$-orbit of $\hat \T_1$
and its limit leaves.

\begin{lem}\label{combinat}
  There are disjoint and unlinked critical chords $\hat \mc$ and $\hat
  \md$ such that
  \[\si_3^{k_{2n}}(\hat \T_1) \to \hat \mc \text{ and }
  \si_3^{k_{2n+1}}(\hat \T_1) \to \hat \md.\]
\end{lem}

\begin{proof}
  We will first find $\hat \mc$. Let $(k_{n_i})_{i=1}^\infty$ denote
  a subsequence such that $\{\hat \T_{k_{n_i}}\,\mid\,i \ge 0\}$ are
  on the same side of $\mc$ in $\La_\infty$. Since $h$ preserves
  order on $A_\T$, it follows that $\si_3^{k_{n_i}}(\hat \T_1)$
  converges to a chord $\hat \mc$. Indeed, since $\si_3^m(\hat
  \T_1)$ are pairwise disjoint convex sets, a subsequence of such
  sets can only converge to a chord or a point in the circle.
  However, by construction on either side of $\mc$ in the
  combinatorial model there are images of $\T_1$. Hence the limit of
  $\si_3^{k_{n_i}}(\hat \T_1)$ must be a chord $\hat \mc$. In order
  to see that $\hat\mc$ is critical we will use that the edges of
  triangles under which $\si_3(\hat\mc)$ is located require more and
  more time to cover $\T_1$ and, hence, $\si_3(\hat\mc)$ must be a
  point as desired. Now we will implement this idea.

   Suppose for contradiction that $\hat \mc$ is not
  critical. Then a non-degenerate interval $I$ lies under $\si_3(\hat \mc)$,
  so there is a minimal $N > 0$ such that $\si_3^N(I)$ contains
  $\T_1$.  This implies that one of the first $N$ images of the short
  edge of $\T_{k_{n_i}}$ is neither short nor empty in $\La_{n_i+1}$,
  with the interval underneath either mapping out of order or
  containing $\T_1$.  This contradicts the details of the
  construction; recall that $p_{n_i}-k_{n_i+1} \to \infty$ so we can choose $i$ so that
  $p_{n_i}-k_{n_i+1}>N$.  Then
  the short edge of $\T_{k_{n_i}}$ is short on the segment of its
  orbit from $k_{n_i}+1$ to $p_{n_i}$ in $\La_{n_i}$ and is empty on
  the segment from $p_{n_i}+1$ to $k_{n_i+1}$ in the lamination
  $\La_{n_i+1}$, and therefore does not contain $\T_1$.  Therefore,
  $\hat \mc$ is a critical leaf.

  Let $\hat \md$ be a critical chord obtained similarly from
  $\T_{k_{2n+1}}$.  In principle, it is possible that four chords
  leaves can arise this way: two for each side of $\mc$ and $\md$.
  However, since $g_\infty(\mc) \neq g_\infty(\md)$ and $h$ preserves
  order, it follows that $\si_3(\hat \mc) \neq \si_3(\hat \md)$.
  There is no room for another critical chord unlinked with $\hat \mc$ and
  $\hat \md$, so $\T_{k_{2n}} \to \hat \mc$ and $\T_{k_{2n+1}} \to \hat
  \md$.
\end{proof}

Let $\hat \Ta$ denote the critical portrait $\{\hat \mc, \hat \md\}$.

\begin{lem}\label{lem:aperiodic}
  The critical portrait $\hat \Ta$ has aperiodic kneading.
\end{lem}

\begin{proof}
  By Lemma \ref{lem:extend_lam} for each $n$ the first segment of
  length $p_{2n+1}-p_{2n}$ of the itinerary of $\si_3(\hat \mc)$
  equals the itinerary of $\T_1$ from $\T_{p_{2n}}$ to $\T_{p_{2n+1}}$
  as given by $\La_{2n+1}$. By construction, as we vary $n$, the
  initial segments of the itinerary of $\si_3(\hat \mc)$ will have to
  contain all itineraries $\tau^\mc_i$, so the itinerary of
  $\hat \mc$ cannot be periodic.
\end{proof}

According to Theorem~\ref{kiwistruct}, there is a unique $\hat
\Ta$-compatible lamination $\sim_{\hat \Ta}=\sim$.  The quotient
$J_\sim = \uc / {\sim}$ is a dendrite, with corresponding quotient map
$p:\uc \to J_\sim$.  Denote by $\La_\sim$ the collection of (maybe
degenerate) boundary chords of convex hulls of all $\sim$-classes.
As in \cite{thu85} the set $\La_\sim$ can be interpreted
geometrically. For each $\hat\Ta$-itinerary $\tau$ there are unique
preimages of $\hat \mc$ and $\hat \md$ which are pullbacks
corresponding to $\tau$.  Note that the closures of the preimages of
$\hat \mc$ and $\hat \md$ generate a $\hat\Ta$-compatible
lamination, which equals $\La_\sim$ by the uniqueness of the $\hat
\Ta$-compatible lamination (Theorem~\ref{kiwistruct}); it is
therefore not difficult to see that preimages of $\hat \mc$ and
$\hat \md$ are dense in $\La_\sim$.

\begin{lem}\label{lem:wander}
  The triangle $\hat \T_1$ is a $\sim$-class which is a wandering triangle, $\hat
  \mc$ and $\hat \md$ are the critical $\sim$-classes, and the forward
  orbit of $p(\hat \T_1)$ is condense in $J_\sim$.
\end{lem}

\begin{proof}
  By Theorem~\ref{kiwistruct}, $\hat \T_1$ is contained in a finite
  $\sim$-class $\mathbf w$. Since $\hat \T_1$ is wandering, so is
  $\mathbf w$.
  By \cite[Theorem A]{blolev99} (see also \cite{kiwigap}), it follows that $\hat
  \T_1=\mathbf w$ and the critical $\sim$-classes are $\hat \mc$ and $\hat \md$.

  That the forward orbit of $p(\hat \T_1)$ is condense in
  $J_\sim$ is equivalent to the forward orbit of $\hat
  \T_1$ being dense in $\La_\sim$ in the sense that every leaf in $\La$
  can be approximated arbitrary well by an edge of some $\hat \T_N$. As discussed above,
  it suffices to show that every
  critical pullback is approximated by the orbit of the triangle $\hat
  \T_1$ arbitrarily well. Choose a precritical itinerary $\tau$ of
  length $N$ with the last entry $\hat \mc$ and let $\hat \ell_\tau$
  be the pullback leaf of $\hat \mc$ in $\La_\sim$ with itinerary
  $\tau$. Since $\hat \mc$ is a $\sim$-class, so is $\hat \ell_\tau$.

By construction, for every pullback $\tilde\mc$ of $\mc$ in the
combinatorial lamination there is a triangle ``close'' to it which
follows the itinerary of $\tilde\mc$ up to the point that
$\tilde\mc$ maps to $\mc$. This will imply that in the
$\si_3$-implementation of this combinatorial lamination the
corresponding triangle can be chosen arbitrary close to the
corresponding critical pullback of $\hat \mc$. The formal proof of
this fact is given below.

  Choose a sequence of precritical itineraries $\tau^\mc_{u_i}$, whose
  lengths approach infinity, all of which in the end have a segment
  coinciding with $\tau$. By construction and
  Lemma~\ref{lem:extend_lam}, in the lamination $\La_{2u_i}$ there is
  a leaf $\ell(\La_{2u_i-1})$ which is a pullback of $\mc$ exhibiting
  itinerary $\tau^\mc_{u_i}$ right before it maps to $\mc$.  This
  implies that the appropriate image of $\ell(\La_{2u_i-1})$ is the
  leaf $\ell_\tau$ with itinerary $\tau$. Moreover, the last triangle
  $\T^l(\La_{2u_i-1})$ is adjacent to $\ell(\La_{2u_i-1})$ in
  $\La_{2u_i-1}$, maps to the triangle, denoted here $\T^i$, which is
  adjacent to $\ell^\tau$ and then to the triangle $\T_{k_{2u_i}}$
  adjacent to $\mc$ in $\La_{2u_i}$.

  Let us show that then $\hat \T^i$ converges to $\hat
  \ell_\tau$. Indeed, as in the proof of Lemma~\ref{combinat} we may
  assume that a sequence of triangles $\hat \T^i$ converges to some
  chord $\hat \ell$ from one side. Then $\si_3^N(\hat \T^i)=\hat
  \T_{k_{2u_i}}$ and by continuity $\si_3^N(\hat \ell)$ equals the
  limit of the triangles $\hat \T_{k_{2u_i}}$ which, by
  Lemma~\ref{combinat}, is $\hat \mc$. On the other hand, by
  construction the itinerary of $\hat \ell$ before that moment
  coincides with $\tau$ up to its last entry. Hence $\hat \ell=\hat
  \ell_\tau$ as desired. Since $\tau$ was an arbitrary precritical
  itinerary ending with $\mc$ and since the same arguments can be used
  if it ends with $\md$, we conclude that in fact any pullback leaf of
  a critical leaf in $\La_\sim$ is a limit leaf of the forward orbit
  of $\hat \T_1$, and so, as explained above, the forward orbit of
  $p(\hat \T_1)$ is condense in $J_\sim$.
\end{proof}

\begin{lem}\label{lem:isomorphic_copy}
  $\La$ contains an order isomorphic copy of $(\La_0, g_0)$.
\end{lem}
\begin{proof}
  The preimage of the sequence $(\T_{k_{2n}})$ corresponding to the
  itinerary $\tau$ converges to a leaf $\hat \mc^\tau$, and likewise
  for $(\T_{k_{2n+1}})$ and $\md$.  The collection of all such leaves
  for every itinerary $\tau$ represented in $(\La_0, g_0)$, as well as
  some forward images of $\hat \mc$ and $\hat \md$, then forms
  an order isomorphic copy of $(\La_0, g_0)$ in $\La$.  The
  straightforward details are left to the reader.
\end{proof}

Given a full sequence of end-critical itineraries ending in $\mc$
and a full sequence of end-critical itineraries ending in $\md$, we
can find the critical portrait $\Ta$ from Theorem~\ref{thm:inv_lam}
so that a sublamination of $\sim_\Ta$ is order isomorphic to the
lamination $(\La_\infty, g_\infty)$ constructed on the basis of
these two full sequences of end-critical itineraries by repeated
application of Lemma~\ref{lem:extend_lam}. Hence, we have proven the
following theorem (recall the the notion of condensity was
introduced in Subsection~\ref{ss:condens}.

\begin{thm}\label{thm:inv_lam}
  If $(\La_0, g_0)$ is a $\mc$-lamination, then there exists a
  critical portrait $\hat \Ta$ such that the following hold.
  \begin{enumerate}
  \item $\hat \Ta$ has aperiodic kneading.
  \item $\sim_\Ta$ has a wandering class $\T_1$ consisting of three points.
  \item The orbit of $\T_1$ is condense in the quotient space $\uc /
    \sim_\Ta$.
  \item A subcollection of $\sim$-classes forms a finite dynamical
    sublamination which is conjugate to $(\La_0, g_0)$.
  \end{enumerate}
\end{thm}

\subsection{Locating combinatorial laminations in
  $\A_3$}\label{ss:loca3}

In this section, we show that the collection of critical portraits
corresponding to the $\si_3$-invariant laminations given by
Theorem~\ref{thm:inv_lam} is dense in the space of all critical
portraits with the compact-unlinked topology (see Definition~\ref{cu}).
To do so, we take a critical portrait $\Ta$ and construct a
$\mc$-lamination $(\La, g)$ such that the critical portrait of any $\si_3$-invariant lamination
containing an order isomorphic copy of $(\La, g)$ is close to $\Ta$.
The critical portrait of any $\si_3$-invariant lamination given by
Theorem~\ref{thm:inv_lam} for $(\La, g)$ is then close to $\Ta$.

To do so, we make the observation that certain dynamical behaviors in
invariant laminations tell us where particular points of the circle
can be located.  For example, if there are points $a, b \in \uc$ such that
$\si_3(a) < a < b < \si_3(b)$ in counter-clockwise order on the
interval $[\si_3(a), \si_3(b)]$, we have either $0 \in (a,b)$ or $1/2
\in (a,b)$.  We can use this sort of information to pinpoint the
locations of critical portraits to high precision.  Let $A_\mc(\Ta)$
be the component of $\uc\sm \Ta$ which is an arc whose endpoints are
the endpoints of $\mc$. Define $A_\md(\Ta)$ similarly.  Given a leaf
$\ell$ contained in a connected component $H$ of a $\Ta$-unlinked
class, let $I_\ell$ be the arc in $H$ with the same endpoints.  We
often refer to $I_\ell$ as being \emph{under $\ell$}; note that this
is compatible with the previous definition of ``under a leaf'', but not ``under the edge of a triple''.

\begin{defn}[Settled lamination]\label{d:settl}
  Let $(\La, g)$ be a critical lamination  with two disjoint
  critical leaves $\Ta=\{\mc, \md\}$ such that the following
  conditions hold.
  \begin{enumerate}
  \item At least one of the following conditions hold.
    \begin{enumerate}
    \item There exists a leaf $\ell_x \in \La$ such that $I_{\ell_x}
      \subset I_{g(\ell_x)} \subset A_\mc(\Ta)$.  (This models the
      case where $I_{\ell_x}$ is a short interval containing $0$ or $1/2$.)
    \item There exists a leaf $\ell_x \in \La$ such that $I_{\ell_x}
      \subset I_{g^2(\ell_x)} \subset A_\mc(\Ta)$ and $I_{g(\ell_x)}
      \subset A_\md(\Ta)$.  (This models the case where $I_{\ell_x}$
      is a short interval containing $1/4$ or $3/4$.)
    \end{enumerate}
  \item $(\La, g)$ contains disjoint preimages $\ell_y$ and $\ell_z$
    of $\ell_x$, each contained entirely within a connected component
    of a $\Ta$-unlinked class not coinciding with the $\Ta$-unlinked class
    which contains $\ell_x$.
  \item The finite orbits of $\mc$ and $\md$ in $\La$ are
    \emph{disjoint} from $I_{\ell_x}\cup I_{\ell_y}\cup I_{\ell_z}$
    and do not contain periodic points.
  \item Every class of $\La$ is either an image or a preimage of $\mc$
    or $\md$.
  \end{enumerate}
  Then we call $(\La, g)$ a finite \textbf{settled} lamination.
\end{defn}

The advantage of settled laminations is that for them one can define
not only $\La$-itineraries but also itineraries with respect to the
hypothetical points $x$, $y$, and $z$ represented by $I_x$, $I_y$, and
$I_z$. This can be used to locate, with any given precision, a cubic
critical portrait whose geometric lamination contains an order
isomorphic copy of $(\La, g)$. Denote by $R_{1/2}:\uc\to \uc$ the
rotation $R_{1/2}(x)=x+1/2 \mod 1$.

\begin{lem}\label{lem:init_cond1}
  Let $\Ta'$ be a cubic critical portrait and $U$ be a neighborhood of
  $\Ta'$ in the compact-unlinked topology. Then there exists a finite
  settled lamination $(\La, g)$ with critical classes $\mc$ and $\md$
  such that, if $\Ta''$ is the critical portrait of a lamination containing an order isomorphic
  copy of $(\La, g)$, then either $\Ta'' \in U$ or $R_{1/2}(\Ta'') \in U$.
  Moreover, $(\La, g)$ can be
  continued to a $\mc$-lamination.
\end{lem}

\begin{proof}
  Let $\Ta'=\{\mc', \md'\}$.
  There are cases to consider: either $I_{\mc'}$ contains a fixed point
  (equivalently $\si_3(\mc') \notin I_{\mc'}$), or $I_{\md'}$ contains a
  fixed point (equivalently $\si_3(\md') \notin I_{\md'}$), or $I_{\mc'}$
  and $I_{\md'}$ each contain a period two point and no fixed point
  (equivalently $\si_3(\mc') \in I_\mc'$ and $\si_3(\md') \in I_{\md'}$).
  We consider the case that $I_\md'$ contains 0; without loss of
  generality, the only other case is the third, and we leave its
  consideration to the reader.

  Given $t\in \uc$, let $W(t)=(w_0(t), w_1(t), \dots)$ be a sequence
  of arcs $(0, 1/3)$, $(1/3, 2/3)$, $(2/3, 0)$ or points $0, 1/3, 2/3$
  such that $\si_3^j(t)\in w_j(t)$ for all $j\ge 0$. The initial
  segment of $W(t)$ of length $k$ is denoted by $W_k(t)$.

  It is easy to see that there are arbitrarily large numbers $N>M$ and a
  critical portrait $\Ta =\{\mc, \md\}$ with $\mc, \md$ disjoint which
  have the following properties.

  \begin{enumerate}
  \item[(i)] The leaves $\mc$ and $\md$ are disjoint from $\{k/12,
    k=0, \dots, 11\}$.
  \item[(ii)]  $\{\si_3^N(\mc), \si_3^N(\md)\} = \{1/3, 2/3\}$.
  \item[(iii)] Any critical portrait $\Ta''$ with two critical leaves
    $\mc'', \md''$ such that $W_M(\md'')=W_M(\md)$ and
    $W_M(\mc'')=W_M(\mc)$ belongs to $U$.
  \end{enumerate}

  From now on we fix $\Ta$. Since the endpoints of $\Ta$ are strictly
  preperiodic, Theorem~\ref{kiwistruct} implies that $\sim=\sim_\Ta$
  (defined in Definition~\ref{lamtheta}) has several properties.  
  In particular, $\{0\}=\mx, \{\si_3^N(\md)\}=\my$ and
  $\{\si_3^N(\mc)\}=\mz$ are $\sim$-buds.

  Also, $\mc$ and $\md$ are $\sim$-classes.  To see this, suppose
  $\mg$ is the $\sim$-class containing $\mc$.  Since $\si_3^N(\mc) =
  \mz$ is a bud, we have that $\si^N(\mg) = \mz$, so $\mg$ is finite
  and non-periodic.  If $\si_3(\mg)$ is not degenerate, then it is at
  some point critical, and hence contains $\md$.  We then see that
  $\mg$ eventually maps to $\si_3^N(\md)=2/3$, but this is a
  contradiction, as neither $1/3$ nor $2/3$ ever maps to the other.

  \begin{figure}
    \centering
    \includegraphics{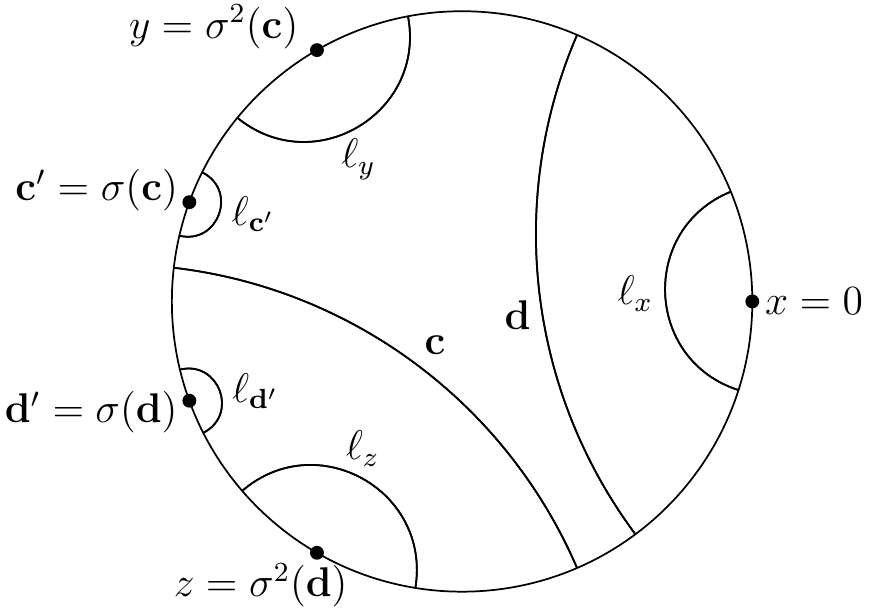}
    \caption{A $\mc$-lamination constructed as in the proof of
      Lemma~\ref{lem:init_cond1}, with $M=1$ and $N=2$.}
    \label{fig:init_cond1}
  \end{figure}

  We will now define the required settled lamination by identifying a
  finite sublamination of the lamination $\sim$. Thus, in what follows
  we will consider preimages and pullbacks of leaves in the sense of
  the lamination $\sim$. Let $\md_{-1}, \ldots, \md_{-N}$ denote
  repeated preimages of $\md$ into $A_{\md}$.  Note then that $0 \in
  I_{\md_{-N}} \subset \ldots \subset I_{\md_{-1}} \subset A_{\md}$.  It
  is also apparent that $0$ is the only point of the forward orbits of
  $\mc$ and $\md$ which lies in $I_{\md_{-N}}$ (any forward image of
  $\mc$ or $\md$ maps to
  $0$ in at most $N$ steps and no point in $I_{\md_{-N}}$ does so except for $0$).  We set $\ell_x = \md_{-N}$
  and choose preimages $\ell_y$ and $\ell_z$ of $\ell_x$ to satisfy $\my
  \in I_{\ell_y}$, $\mz \in I_{\ell_z}$.

  Let us show that $\ell_y$ is contained in one component of the
  $\Ta$-unlinked class $Q$ containing $\my$. Indeed, by way of
  contradiction suppose that $\ell_y$ intersects two components of
  $Q$. Then $\ell_y$ separates $\mc$ from $\md$. Hence
  $\si_3(\ell_y)=\ell_x$ separates $\si_3(\mc)$ from $\si_3(\md)$.
  Therefore, $I_{\ell_x}$ contains either $\si_3(\mc)$ or
  $\si_3(\md)$, a contradiction. Moreover, no point of the
  forward orbits of $\mc$ and $\md$ enters $I_{\ell_y} \cup
  I_{\ell_z}$ except for $\my$ and $\mz$.

  Since $\si^{N+1}_3(\mc)=\si^{N+1}_3(\md)=0$ and $N>M$, we can pull
  back the leaves $\ell_y$ and $\ell_z$ along the branches of the
  backward orbit of $0$ towards $\si_3^M(\md)$ and $\si_3^M(\mc)$
  respectively. Denote by $I(\si_3^{M+i}(\md))$ the pullback of
  $I_{\ell_y}$ which contains the point $\si_3^{M+i}(\md)$, and by
  $\ell(\si_3^{M+i}(\md))$ the $\sim$-leaf which has the same
  endpoints as $I(\si_3^{M+i}(\md))$. Observe that the only point of
  the orbits of $\mc, \md$ in $I(\si_3^{M+i}(\md))$ is
  $\si_3^{M+i}(\md)$. Similarly, denote by $I(\si_3^{M+i}(\mc))$ the
  pullback of $I_{\ell_z}$ which contains the point
  $\si_3^{M+i}(\mc)$, and by $\ell(\si_3^{M+i}(\mc))$ the $\sim$-leaf
  which has the same endpoints as $I(\si_3^{M+i}(\mc))$. Here $0\le
  i\le N-M$ so that $I(\ell_y)=I(\si_3^N(\md))=I(y)$,
  $I(\ell_z)=I(\si_3^N(\mc))=I(z)$. Put $I(x)=I_{\ell_x}$.

  Set
  \begin{multline*} V = I(\si^M_3(\md)) \cup I(\si^{M+1}_3(\md)) \cup
    \ldots \cup I(y) \cup \\ I(\si^M_3(\mc)) \cup I(\si^{M+1}_3(\mc))
    \cup \ldots \cup I(z) \cup I(x)
  \end{multline*}

  Note that all the intervals comprising this union are pairwise
  disjoint.  Consider the critical lamination $\La$ defined as

\[\La=\{\si^i_3(\mc), \si^i_3(\md)\}_{i=0}^M\cup \{\md_{-j}\}_{j=1}^N\cup
\{\ell(\si_3^{M+i}(\mc)), \ell(\si_3^{M+i}(\md))\}_{i=0}^{N-M}.\]

Define $\bLa$ as $\La\sm \{\si_3^M(\mc), \si_3^M(\md)\}$ and set
$g=\si_3|_{\bLa}$. This defines the settled dynamical lamination
$(\La, g)$ whose only last classes are $\si_3^M(\mc)=\mc',
\si_3^M(\md)=\md'$. Moreover, by construction $\ell(\mc')$ is a
preimage of $\md$ which is adjacent to $\mc'$ in $(\La, g)$
and which maps to $\md_{-N}$ by $\si_3^{N+1-M}$ and then to $\md$ by
$\si_3^{2N+1-M}$.

Consider a lamination $\sim$ with critical portrait $\Ta''$ containing
an order isomorphic copy of $(\La, g)$. We need to show that then
$\Ta''\in U$ or $R_{1/2}(\Ta'')\in U$. Set $\Ta''=\{\mc'',
\md''\}$. Let $h$ be the order isomorphism between $(\La, g)$ and the
appropriate finite sublamination $(\La'', \si_3|_{\bLa''})$ of
$\La_\sim$ so that $\md''=h(\md), \mc''=h(\mc)$. The dynamics of $g$
(and hence of $\si_3$ on $\La''$) implies that there is a
$\si_3$-fixed point in $h(I(x))$; without loss of generality we may
assume that this fixed point is $0$ (otherwise we will apply $R_{1/2}$
to $\sim$). Moreover, $h(I(y))$ must contain one $\si_3$-preimage of
$0$ not equal to $0$ while $h(I(z))$ contains the other
$\si_3$-preimage of $0$ not equal to $0$. Since $h$ is an order
isomorphism and $g=\si_3|_{\bLa}$, it follows that in fact $I(y)$
contains the same preimage of $0$ as $h(I(y))$, and $I(z)$ contains
the same preimage of $0$ as $h(I(z))$.

Now, the fact that $(\La, g)$ and $(\La'', \si_3|_{\bLa''})$ are order
isomorphic implies that $W_M(\mc)=W_M(\mc'')$ and
$W_M(\md)=W_M(\md'')$. By the choice of $M$ this implies that
$\Ta''\in U$ as desired.

Finally, we note that it is possible to continue $(\La, g)$ to a
$\mc$-lamination.  Indeed, $\mc'$ and $\md'$ are both adjacent to
preimages of $\md$. Moreover, it follows from the above that
the arc under $\ell(\md')$ maps onto the entire $\uc$
in the one-to-one fashion (except for its endpoints) by
$\si_3^{2N+2-M}$. This allows us to find
the $\si_3^{2N+2-M}$-preimage $\mm$ of $\mc$ adjacent to $\md'$ which does not map
by $\si_3$ under the leaf $\mell$ adjacent to $\mc'$ until it ($\mm$) maps to $\mc$ and
then to $\mc'$.
One then adds a triple $\T_1$
leaf-like and adjacent to both $\md'$ and to $\mm$ very close to $\mm$
so that the appropriate initial segment of the orbit of $\T_1$ follows the orbit of
$\mm$ until $\mm$ maps to $\mc$ and then to $\mc'$. At this moment the image of $\T_1$
is leaf-like adjacent to $\mc'$ and to $\mell$. 
The result is a $\mc$-lamination.
\end{proof}

We can now combine all of these ingredients to give a proof of our
main result.  Recall that $\WT_3$ is the family of all cubic critical
WT-portraits.

\begin{mathe}
  For each open $U\subset \A_3$ there is an uncountable set $\B
  \subset U\cap \AP_3 \cap \WT_3$ of critical portraits $\Theta$
  such that the following facts hold:

  \begin{enumerate}

  \item there exists a wandering branch point in $J_{\sim_\Ta}$ whose
    orbit is condense in $J_{\sim_\Ta}$;

  \item all maps $f_{\sim_\Ta}|_{J_{\sim_\Ta}}, \Ta\in \B,$ are
    pairwise non-conjugate;

  \item for each $\Ta\in \B$ there exists a polynomial $P_\Ta$ such
    that $P_\Ta|_{J_{P_\Ta}}$ is conjugate to
    $f_{\sim_\Ta}|_{J_{\sim_\Ta}}$.

  \end{enumerate}

\end{mathe}

\begin{proof}
  Fix an open set $U\subset \A_3$ and full sequences of precritical
  itineraries $(\tau_n^\mc)_{n=1}^\infty$,
  $(\tau_n^\md)_{n=1}^\infty$ ending in $\mc$ and $\md$
  respectively. By Lemma~\ref{lem:init_cond1} for the neighborhood
  $U$ there exists a finite settled lamination $(\La_0, g_0)$ and a
  continuation $(\La'_0, g'_0)$ of $(\La_0, g_0)$ such that
  $(\La'_0, g'_0)$ is a $\mc$-lamination. Then by
  Theorem~\ref{thm:inv_lam} there is a $\si_3$-invariant lamination
  $\sim$ continuing $(\La'_0, g'_0)$ satisfying (1).

  Let now prove claim (2) of the theorem. Let $\hat\Theta$ be the
  critical portrait of $\sim$.  Since $\sim$ continues $(\La_0,
  g_0)$, it follows by Lemma~\ref{lem:init_cond1} that $\hat \Theta
  \in U \cap \AP_3 \cap \WT_3$. Consider the set $\hat \B$ of all critical WT-portraits
  constructed in this way. It remains to show that there is an
  uncountable subcollection of $\hat \B$ such that the corresponding
  topological polynomials are pairwise non-conjugate. Recall that
  the construction of laminations is done on the basis of
  Lemma~\ref{lem:extend_lam} (see the comment after
  Lemma~\ref{lem:extend_lam}). Since the images of the points
  $\si_3(\hat \mc), \si_3(\hat \md)$ shadow longer and longer
  segments of the forward orbit of $\hat \T_1$, then the $\hat
  \Ta$-itineraries of $\hat \mc$ and $\hat \md$ are
  \emph{completely} defined by the behavior of the triangle $\hat
  \T_1$. On the other hand, the behavior of the triangle $\hat \T_1$
  is completely defined by the sequences of itineraries
  $(\tau_n^\mc)_{n=1}^\infty$ and $(\tau_n^\md)_{n=1}^\infty$. Hence
  the $\hat \Ta$-itineraries of $\hat \mc$ and $\hat \md$ are
  completely determined by by the sequences of itineraries
  $(\tau_n^\mc)_{n=1}^\infty$ and $(\tau_n^\md)_{n=1}^\infty$.

  Suppose that two pairs $X$ and $Y$ of full sequences of
  itineraries differ, and denote the critical portraits $\Ta_X=(\hat
  \mc_X, \hat \md_X)$ and $\Ta_Y=(\hat \mc_Y, \hat \md_Y)$ as
  constructed above. Let us show that the itineraries of $\si_3(\hat
  \mc_X)$ and $\si_3(\hat \mc_Y)$ in their respective laminations
  are distinct. Without loss of generality, suppose that for some
  integer $n$ the itineraries of $g^X_n(\mc_X)$ and $g^n_Y(\mc_Y)$
  coincide until a certain moment when the following holds:

  \begin{enumerate}

  \item the triangles $\T^l(\La_n^X, g_n^X)$ and $\T^l(\La_n^Y,
  g_n^Y)$ are adjacent to $g^n_X(\mc_X)$ and $g^n_X(\mc_Y)$, respectively;

  \item the triangles $\T^l(\La_n^X, g_n^X)$ and $\T^l(\La_n^Y,
  g_n^Y)$ are also adjacent to $\mell(\La_n^X, g_n^X)$ and $\mell(\La_n^Y, g_n^Y)$
  which they will shadow according to the construction;

  \item the itineraries of $\mell(\La_n^X, g_n^X)$ and $\mell(\La_n^Y, g_n^Y)$
  are distinct because they were constructed using distinct
  itineraries from $X$ and $Y$.
  \end{enumerate}

  Hence, the itineraries of $\si_3(\hat \mc_X)$ and $\si_3(\hat
  \mc_Y)$ are also distinct. This implies that if $X\ne Y$, then
  $\Ta_X\ne \Ta_Y$ and the critical itineraries of $\Ta_X$ and
  $\Ta_Y$ are distinct. Therefore in any open set in $\A_3$ there
  are uncountably many critical WT-portraits (as there are
  uncountably many distinct pairs of full sequences of itineraries
  with each itinerary ending in $\mc$ or $\md$). In terms of their
  topological polynomials, $f_{\Ta_X}$ and $f_{\Ta_Y}$ may yet be
  conjugate, but can only differ by the labeling of their critical
  leaves.  There is therefore an uncountable subset of critical
  portraits in $U \cap \A_3 \cap \WT_3$ whose induced topological
  polynomials on their topological Julia sets are pairwise
  non-conjugate. This completes the proof of (2).

  To prove (3) observe that by Kiwi's results (see
  Theorem~\ref{kiwistruct}) for each topological polynomial $f$
  constructed above there exists a complex polynomial $P$ which is
  monotonically semiconjugate to the $f$ on its Julia set. By (1)
  all these topological polynomials $f$ have points with condense
  orbits in their topological Julia sets. Therefore by
  \cite[Theorem 3.6(1)]{bco11} $P$ is conjugate to $f$ as desired.
\end{proof}

\end{document}